\documentclass{lmcs}
\pdfoutput=1
\usepackage[utf8]{inputenc}

\usepackage{lastpage}
\lmcsdoi{21}{3}{7}
\lmcsheading{}{\pageref{LastPage}}{}{}%
{Feb.~29,~2024}{Jul.~23,~2025}{}

\usepackage[hidelinks,pdfusetitle]{hyperref}
\usepackage{amsmath,amssymb,amsfonts}
\usepackage{cleveref}
\usepackage{tikz}
\usetikzlibrary{decorations.markings,arrows.meta}
\usetikzlibrary{nfold}
\usepackage{tikz-cd}
\usepackage[only,llbracket,rrbracket]{stmaryrd}
\usepackage{theorems}
\usepackage{macros}
\usepackage{fixes}

\title{Rewriting techniques for relative coherence}
\author[S.~Mimram]{Samuel Mimram\lmcsorcid{0000-0002-0767-2569}}
\address{LIX, CNRS, École polytechnique, Institut Polytechnique de Paris, 91120 Palaiseau, France}
\email{samuel.mimram@polytechnique.edu}

\newcommand{\enumlabel}[1]{#1}

\renewcommand{\TODO}[1]{}

\begin{document}

\begin{abstract}
  A series of works has established rewriting as an essential tool in order to prove coherence properties of algebraic structures, such as MacLane's coherence theorem for monoidal categories, based on the observation that, under reasonable assumptions, confluence diagrams for critical pairs provide the required coherence axioms. We are interested here in extending this approach simultaneously in two directions. Firstly, we want to take into account situations where coherence is partial, in the sense that it only applies to a subset of the structural morphisms. Secondly, we are interested in structures which are cartesian in the sense that variables can be duplicated or erased. We develop theorems and rewriting techniques in order to achieve this, first in the setting of abstract rewriting systems, and then extend them to term rewriting systems, suitably generalized to take coherence into account. As an illustration of our results, we explain how to recover the coherence theorems for monoidal and symmetric monoidal categories.

\end{abstract}

\maketitle


\section{Introduction}
\subsection{Coherence results}
Coherence results are fundamental in category theory. They can be seen both as a way of formally simplifying computations, by ensuring that we can consider strict algebraic structures without loss of generality, and as a guide for generalizing computations, by ensuring that we have correctly generalized algebraic structures in higher dimensions and taken higher-dimensional cells in account. Such results have been obtained for a wide variety of algebraic structures on categories, including monoidal categories~\cite{maclane1998categories,maclane1963natural}, symmetric monoidal categories~\cite{maclane1998categories,maclane1963natural}, braided monoidal categories~\cite{joyal1993braided}, tortile monoidal categories~\cite{shum1994tortile}, symmetric monoidal closed categories~\cite{JAY1990271,kelly1971coherence,voreadou1977coherence}, biclosed monoidal (multi)categories~\cite{lambek69}, compact closed categories~\cite{kelly1980coherence}, cartesian closed categories~\cite{babaev1982coherence,mints1982simple}, rig categories~\cite{laplaza1972coherence}, weakly distributive categories~\cite{blute1996natural}, bicategories and pseudo-functors~\cite{lucas2017coherence}, cartesian closed bicategories~\cite{fiore2021coherence}, tricategories~\cite{gordon1995coherence,gurski2013coherence}, etc.

The coherence results are often quickly summarized as ``all diagrams commute''. However, this is quite misleading~\cite{kelly1972abstract}: firstly, we only want to consider diagrams made of structural morphisms, and secondly, we actually usually want to consider only a subset of those diagrams.
In the general case, the goal is thus to identify a class of diagrams which commute or, better, decide equality between structural morphisms and provide explicit and combinatorial descriptions of those.
One of the goals of this article is to clarify the situation and the relationships between the various approaches.

\subsection{Coherence from rewriting theory}
A field which provides many computational techniques to show that diagrams commutes is rewriting~\cite{baader1999term,bezem2003term}. Namely, when a rewriting system is terminating and locally confluent, which can be verified algorithmically by computing its critical pairs, it is confluent and thus has the Church-Rosser property, which implies that any two zig-zags can be filled by local confluence diagrams. By properly extending the notion of rewriting system with higher-dimensional cells in order to take coherence into account (those cells specifying which confluence diagrams commute), one is then able to show coherence results of the form ``all diagrams commute'' up to the coherence laws.
This idea of extending rewriting theory in order to handle coherence dates back to pioneering work from people such as Newman~\cite{newman42} (who observed that the 1-complex associated to a convergent abstract rewriting system can be made into a simply connected 2-complex by adding 2-cells corresponding to confluence diagrams), Squier~\cite{lafont1995new,squier1994finiteness} (who observed that the confluence diagrams for critical pairs of a convergent monoid presentation generate the full congruence on the monoidal category associated to the rewriting system), or MacLane~\cite{maclane1963natural} (whose original proof for the coherence theorem for monoidal categories is implicitly based on rewriting techniques). Furthermore, the work of Power~\cite{power1989abstract} and Street~\cite{street1996categorical} have reformulated rewriting in categorical terms, and paved the way for higher-dimensional generalizations of rewriting in the context of computads~\cite{street1976limits}, which are also known as polygraphs~\cite{polygraphs,burroni1993higher}, and have been used to recover various coherence theorems~\cite{guiraud2012coherence}, that can often be interpreted as computing (polygraphic) resolutions in suitable settings~\cite{polygraphs}. More recently, these generalizations were also adapted to homotopy type theory~\cite{kraus2021rewriting}.
We would also like to mention here that this trend of work has motivated extensions of rewriting to various 2-dimensional settings, which should certainly be helpful at some point in relation to coherence: monoidal categories~\cite{lafont2003towards}, and hypergraph categories which allow rewriting modulo symmetries or Frobenius structures~\cite{DBLP:journals/jacm/BonchiGKSZ22,DBLP:journals/mscs/BonchiGKSZ22,DBLP:journals/mscs/BonchiGKSZ22a}.

\subsection{Coherence for monoidal categories}
One of the first and most important instance of a coherence theorem is the one for monoidal categories, originally due to Mac Lane. Since it will be used in the following as one of the main illustrations, we begin by recalling it here, and discuss its various possible formulations.

A \emph{monoidal category} consists of a category~$C$ equipped with a tensor bifunctor and unit object respectively noted
\begin{align*}
  \otimes:C\times C&\to C
  &
  e:1&\to C
\end{align*}
together with natural isomorphisms
\begin{align*}
  \alpha_{x,y,z}:(x\otimes y)\otimes z&\to x\otimes(y\otimes z)
  &
  \lambda_x:e\otimes x&\to x
  &
  \rho_x:x\otimes e&\to x
\end{align*}
called \emph{associator} and \emph{left} and \emph{right unitors}, satisfying
two well-known axioms stating that the diagrams
\[
  \begin{tikzcd}[column sep=50]
    ((x\otimes y)\otimes z)\otimes w\ar[dd,"\alpha_{x\otimes y,z,w}"']\ar[r,"\alpha_{x,y,z}\otimes w"]&(x\otimes (y\otimes z))\otimes w\ar[dr,"\alpha_{x,y\otimes z,w}"]\\
    &&x\otimes ((y\otimes z)\otimes w)\ar[d,"x\otimes\alpha_{y,z,w}"]\\
    (x\otimes y)\otimes (z\otimes w)\ar[rr,"\alpha_{x,y,z\otimes w}"]&&x\otimes (y\otimes (z\otimes w))
  \end{tikzcd}
\]
and
\[
  \begin{tikzcd}
    (x\otimes e)\otimes y\ar[dr,"\rho_x\otimes y"']\ar[rr,"\alpha_{x,e,y}"]&&\ar[dl,"x\otimes\rho_y"]x\otimes(e\otimes y)\\
    &x\otimes y
  \end{tikzcd}
\]
commute for any objects $x$, $y$ and $z$ of~$C$.

Thanks to these axioms, the way tensor expressions are bracketed does not really
matter: we can always rebracket expressions using the structural morphisms
$\alpha$, $\lambda$ and $\rho$, and any two ways of rebracketing an expression
into the other are equal. In fact, and this is an important point in this
article, there are various ways to formalize this~\cite{nlab:mon-coh}:
\newcommand{\mctc}{(M1)}
\newcommand{\mcfd}{(M2)}
\newcommand{\mcst}{(M3)}
\newcommand{\mcae}{(M4)}
\begin{enumerate}[(M3)]
\item[\enumlabel{\mctc}] Every diagram in a free monoidal category made up of~$\alpha$, $\lambda$ and $\rho$ commutes\\\cite[Corollary 1.6]{joyal1993braided}, \cite[Theorem VI.2.1]{maclane1998categories}.
\item[\enumlabel{\mcfd}] Every diagram in a monoidal category made up of~$\alpha$, $\lambda$ and $\rho$ commutes\\
  \cite[Theorem 3.1]{maclane1963natural}, \cite[Theorem XI.3.2]{maclane1998categories}.
\item[\enumlabel{\mcst}] Every monoidal category is monoidally equivalent to a strict monoidal category\\
  \cite[Corollary 1.4]{joyal1993braided}, \cite[Theorem XI.3.1]{maclane1998categories}.
\item[\enumlabel{\mcae}] The forgetful 2-functor from strict monoidal categories to monoidal categories has a left adjoint and the components of the unit are equivalences.
\end{enumerate}
Condition \mcfd{} implies \mctc{} as a particular case and the converse implication can also be shown, so that the two are equivalent. Condition \mcae{} implies \mcst{} as a particular case, and it can be shown that \mcst{} in turn implies \mcfd{}, see~\cite[Theorem XI.3.2]{maclane1998categories}.

\subsection{Coherence for symmetric monoidal categories}
\label{smc-def}
Although fundamental, taking the previous example of a coherence theorem as a guiding example can be misleading as it hides the fact that the coherence results are in general more subtle: usually, we do not want all the diagrams made of structural morphisms to commute. In order to illustrate this, let us consider the following variant of monoidal categories.

A \emph{symmetric monoidal category} is a monoidal category equipped with a natural transformation
\[
  \gamma_{x,y}:x\otimes y\to y\otimes x
\]
called \emph{symmetry} such that the diagrams
\[
  \begin{tikzcd}
    &y\otimes x\ar[dr,"\gamma_{y,x}"]&\\
    x\otimes y\ar[ur,"\gamma_{x,y}"]\ar[rr,"\id_{x\otimes y}"']&&x\otimes y
  \end{tikzcd}
  \qquad\qquad
  \begin{tikzcd}
    x\otimes e\ar[dr,"\rho_x"']\ar[rr,"\gamma_{x,e}"]&&\ar[dl,"\lambda_x"]e\otimes x\\
    &x
  \end{tikzcd}
\]
and
\[
  \begin{tikzcd}
    &(y\otimes x)\otimes z\ar[r,"\alpha_{y,x,z}"]&y\otimes(x\otimes z)\ar[dr,"y\otimes\gamma_{x,z}"]\\
    (x\otimes y)\otimes z\ar[ur,"\gamma_{x,y}\otimes z"]\ar[dr,"\alpha_{x,y,z}"']&&&y\otimes(z\otimes x)\\
    &x\otimes(y\otimes z)\ar[r,"\gamma_{x,y\otimes z}"']&(y\otimes z)\otimes x\ar[ur,"\alpha_{y,z,x}"']
  \end{tikzcd}
\]
commute for every objects $x$, $y$ and~$z$ of~$C$.

Analogous coherence theorems as above hold and can be formulated as follows:
\begin{enumerate}[(S3)]
\item[(S1)] Every ``generic'' diagram in a (free) symmetric monoidal category made up of~$\alpha$, $\lambda$, $\rho$ and $\gamma$ commutes.
\item[(S2)] Every diagram in a (free) symmetric monoidal category made up of~$\alpha$, $\lambda$ $\rho$ and $\gamma$ commutes precisely when the two sides have the same underlying symmetry\\
  \cite[Corollary~2.6]{joyal1993braided}, \cite[Theorem~XI.1.1]{maclane1998categories}.
\item[(S3)] Every (free) symmetric monoidal category is symmetric monoidally equivalent to a strict symmetric monoidal category~\cite[Proposition~4.2]{may1974e1},
  \cite[Theorem~2.5]{joyal1993braided}.
\item[(S4)] The forgetful 2-functor from strict symmetric monoidal categories to symmetric monoidal categories has a left adjoint and the components of the unit are equivalences.
\end{enumerate}
We can see above that the formulations do not anymore require that ``all diagrams commute''. In order to illustrate why it has to be so, observe that the diagram
\begin{equation}
  \label{gamma-id}
  \begin{tikzcd}
    x\otimes x\ar[r,bend left,"\gamma_{x,x}"]\ar[r,bend right,"\id_{x\otimes x}"']&x\otimes x
  \end{tikzcd}
\end{equation}
does \emph{not} commute in general in monoidal categories, although its morphisms are structural ones. For a concrete example, consider the category of sets and functions equipped with cartesian product as tensor product and $x$ to be any set with at least two distinct elements~$a$ and~$b$. We namely have
\[
  \gamma_{x,x}(a,b)=(b,a)\neq(a,b)=\id_{x\otimes x}(a,b)
  \text.
\]
However, note that the two morphisms do not have the same ``underlying symmetry'' ($\gamma_{x,x}$ corresponds to a transposition, whereas $\id_{x\otimes x}$ to an identity on a $2$-element set). In fact, as stated in (S2), restricting to diagrams where the two morphisms induce the same symmetry is enough to have them always commute. Another way to ensure that the diagrams should commute is to require them to be \emph{generic} (or \emph{linear}) as in (S1), by which we roughly mean that all the objects occurring in the source (or target) object should be distinct: this is not the case in~\cref{gamma-id} since the source object is $x\otimes x$, in which~$x$ occurs twice. Intuitively, this condition ensures that the underlying symmetry of the morphisms is uniquely determined by the positions of the variables, and thus that the diagram commutes as a particular instance of (S1). The same subtlety is implicitly present in the condition (S3): for a strict symmetric monoidal category, we do not require that we the symmetric should be strict (only the associator and unitors, such a category is sometimes also called a \emph{permutative category}~\cite{may1974e1}).




\subsection{Coherence for more general theories}
In order to work in a framework which is able to handle many algebraic structures at once (monoidal categories, symmetric monoidal categories, etc.), we axiomatize the notion of structure using Lawvere 2-theories, which are Lawvere theories enriched in groupoids. In such a theory~$\T$, the 0-cells encode the arities, the 1-cells the operations, and the 2-cells the coherence morphisms, so that a product preserving 2-functor $\T\to\Cat$, \ie an algebra for the theory, corresponds to an actual algebraic structure. In order to perform computations on such a theory, it is often convenient to use a description of it by the means of generators and relations, possibly with good properties, and we use here the notion of term rewriting system, extended in order to account for relations between coherences morphisms (which correspond to zig-zags of rewriting steps).

For ``fully coherent'' structures (\eg monoidal categories) one wants to show that all diagrams made of 2-cells commute, \ie the Lawvere 2-theory contains at most one 2-cell between two given 1-cells. For more general situations (\eg symmetric monoidal categories), one wants to show coherence results relatively to a subtheory $\W\subseteq\T$ (for symmetric monoidal categories, this would be the one generated by $\alpha$, $\lambda$ and $\rho$). Writing $\T/\W$ for the quotient theory (the theory obtained from~$\T$ by turning all the 2-cells of $\W$ into identities), previous conditions can then formulated in this setting in the following way:
\begin{enumerate}[(C3)]
\item[(C1)] Identify a class of pairs of 1-cells $(f_i,g_i)$ such that the hom-categories $\T/\W(f_i,g_i)$ contain at most one morphism.
\item[(C2)] Provide an ``explicit'' description of the quotient theory $\T/\W$.
\item[(C3)] Show that every algebra of~$\T$ is equivalent to an algebra of~$\T/\W$.
\item[(C4)] Show that the forgetful 2-functor from algebras of~$\T/\W$ to algebras of~$\T$ has a left adjoint and the components of the unit are equivalences.
\end{enumerate}
Variants of the condition (C2) have been considered in the literature such as the problem of deciding the equality of 2-cells in~$\T$ (or, equivalently, in free algebras), or provide explicit descriptions of free algebras. An important point is that, in order for the above conditions to make sense, one should first make sure that the quotient is a faithful representation of the original theory (in the sense that the canonical 2-functor $\T\to\T/\W$ is a local equivalence) which, as we show, is the case if and only if $\W$ is \emph{2-rigid} (\cref{prop:rigid-2equivalence}), \ie has at most one 2-cell between any pair of 1-cells.
We illustrate here coherence conditions (C1) and (C2) in the case of symmetric monoidal categories, where they are respectively proved in \cref{smc-S1,smc-S2}.
Coherence condition (C3) and (C4) are respectively formulated as \cref{conj:strict-equiv,conj:strict-adj} and left for future work.

\subsection{Other forms of coherence}
%
There are other possible formulations of coherence, involving what are called \emph{unbiased} variants of the structures.
In the case of monoidal categories, an \emph{unbiased monoidal category} is a category equipped with $n$-ary tensor products for every natural number~$n$, satisfying suitable axioms~\cite[Section~3.1]{leinster2004higher}. The following variant of \mcae{} can then be shown:

\begin{itemize}[(M4')]
\item[(M4')] The forgetful 2-functor from strict monoidal categories to unbiased monoidal categories has a left adjoint and the components of the unit are equivalences.
\end{itemize}

\noindent
This result is in fact a particular instance of a very general coherence theorem due to Power~\cite{power1989general}, see also~\cite{lack2002codescent,shulman2012not}, which originates in the following observation: there is a $2$-monad~$T$ on~$\Cat$ whose strict algebras are strict monoidal categories and whose pseudo-algebras are unbiased monoidal categories. Given a $2$-monad~$T$ on a $2$-category, under suitable assumptions (which are satisfied in the case of the monad of monoidal categories), it can be shown that the inclusion $T\text{-\textbf{StrAlg}}\to T\text{-\textbf{PsAlg}}$ of $2$-categories, from the $2$-category of strict $T$-algebras (and strict morphisms) to the $2$-category of pseudo-$2$-algebras (and pseudo-morphisms) admits a left $2$-adjoint (which can be interpreted as a strictification $2$-functor) such
that the components of the unit of the adjunction are internal equivalences in $T$-pseudo-algebras.

We do not insist much on this general route, as our main concern here is the
relationship with rewriting, which provides ways of handling biased notions of
algebras.

\subsection{Contents of the paper}
We first investigate, in~\cref{sec:coh-cat}, an abstract version of the situation and formally compare the various coherence theorems: we show that quotienting a theory by a subtheory~$\W$ gives rise to an equivalent theory if and only if~$\W$ is coherent (or \emph{rigid}), in the sense that all diagrams commute (\cref{prop:rigid-equivalence}). Moreover, this is the case if and only if they give rise to equivalent categories of algebras (\cref{prop:alg-str}), which can be thought of as a strengthened version of (C4).
We then provide, in \cref{sec:ars}, rewriting conditions which allow showing coherence in practice (\cref{prop:ars-cr}).

Those results are extended, in~\cref{sec:coh-lt}, to the setting of Lawvere 2-theories,
where we are able to axiomatize (symmetric) monoidal categories. One of the main novelties here consists in allowing for coherence with respect to a subtheory~$\W$, which is required to handle coherence for symmetric monoidal categories.
This leads us to conjecture that, when the subtheory is rigid, we always have coherence for algebras (\cref{conj:strict-equiv,conj:strict-adj}).
%
The associated rewriting tools are developed in~\cref{sec:trs}, based on a coherent extension of term rewriting systems (\cref{2trs}), following~\cite{beke2011categorification,cohen2009coherence,malbos2016homological}. In particular, we provide rewriting tools to show that this theory is rigid (\cref{2-coh-rigid}). We also show in \cref{smc-coh} that these tools apply in the case of the theory of symmetric monoidal categories, and use those to recover one of the classical coherence theorems in this setting (\cref{smoncat-strong-coh}).

\Cref{sec:coh-cat,sec:coh-lt} develop the general categorical setting and can be read independently from \cref{sec:ars,sec:trs,smc-coh}, which specifically develop rewriting tools and applications.

This article is an extended version of~\cite{catcoh}, and also corrects a few mistakes unfortunately present there.

\subsection{Acknowledgments}
The author would like to thank the anonymous reviewers for their comments, which improved the paper.

















\section{Relative coherence for categories}
\label{sec:coh-cat}

\subsection{Quotient of categories}
\label{sec:ars-quot}
Fix a category~$\C$ together with a set~$W$ of isomorphisms of~$\C$. Although the situation is very generic, and the following explanation is only vague for now, it can be helpful to think of~$\C$ as a theory describing a structure a category can possess, and~$W$ as the morphisms we are interested in strictifying. For instance, if we are interested in the coherence theorem for symmetric monoidal categories, we can think of the objects of~$\C$ as formal iterated tensor products, the morphisms of~$\C$ as the structural morphisms (the composites of~$\alpha$, $\lambda$, $\rho$ and $\gamma$), and we would typically take $W$ as consisting of all instances of~$\alpha$, $\lambda$ and $\rho$ (but not $\gamma$). This will be made formal in~\cref{sec:coh-lt}.

\begin{definition}
  A functor $F:\C\to\D$ is \emph{$W$-strict} when it sends every morphism of $W$ to an identity.
\end{definition}

\begin{definition}
  \label{quotient-category}
  The \emph{quotient}~$\C/W$ of~$\C$ under~$W$ is the category equipped with a $W$-strict functor $\C\to\C/W$, such that every $W$-strict functor $F:\C\to\D$ extends uniquely as a functor $\tilde F:\C/W\to\D$ making the following diagram commute:
  \[
    \begin{tikzcd}
      \C\ar[d]\ar[r,"F"]&\D\\
      \C/W\ar[ur,dotted,"\tilde F"']
    \end{tikzcd}
  \]
\end{definition}

\noindent
Such quotient categories always exist~\cite{bednarczyk1999generalized}, and we provide below an explicit construction in nicely behaved cases (\cref{prop:rigid-quotient}).
We write $\W$ for the subcategory of~$\C$ generated by~$W$. This subcategory will be assimilated to the smallest subset of morphisms of~$\C$ which contains~$W$ is closed under compositions and identities. The category~$\W$ is a groupoid and it can shown that passing from~$W$ to~$\W$ does not change the quotient.

\begin{lemma}
  The categories $\C/W$ and $\C/\W$ are isomorphic.
\end{lemma}
\begin{proof}
  By definition of quotient categories (\cref{quotient-category}), it is enough to show that the category $\C/W$ is a quotient of~$\C$ by~$\W$. It follows easily from the fact that a functor $\C\to\D$ is $W$-strict if and only if it is $\W$-strict. Namely, the left-to-right implication follows from functoriality and the right-to-left implication from the inclusion~$W\subseteq\W$.
\end{proof}

\noindent
Thanks to the above lemma, we will be able to assume, without loss of
generality, that we always quotient categories by a subgroupoid which has the
same objects as~$\C$.


We will see that quotients are much better behaved when the groupoid we quotient
by satisfies the following property.

\begin{definition}
  A groupoid~$\W$ is \emph{rigid} when any two morphisms $f,g:x\to y$ which are parallel (\ie have the same source, and have the same target) are necessarily equal.
\end{definition}

\noindent
Such a groupoid can be thought of as a ``coherent'' sub-theory of~$\C$: it does not have non-trivial geometric structure in the sense of \cref{lem:rigid}
below.

We will need to use the following properties of categories.

\begin{definition}
  A category is
  \begin{itemize}
  \item \emph{discrete} when its only morphisms are identities,
  \item \emph{contractible} when it is equivalent to the terminal category,
  \item \emph{connected} when there is a morphism between any two objects,
  \item \emph{propositional} when it is a rigid and connected groupoid.
  \end{itemize}
\end{definition}

\begin{lemma}
  \label{contractible-prop}
  A propositional category with an object is contractible.
\end{lemma}
\begin{proof}
  Given a propositional category~$\C$, the terminal functor $\C\to 1$ is full
  (because~$\C$ is connected), faithful (because~$\C$ is rigid) and surjective
  (because~$\C$ has an object).
\end{proof}

\begin{proposition}
  \label{lem:rigid}
  \label{rigid-characterizations}
  Given a groupoid~$\W$, the following are equivalent
  \begin{enumerate}[(i)]
  \item $\W$ is rigid,
  \item $\W$ has identities as only automorphisms,
  \item $\W$ is equivalent to a discrete category,
  \item $\W$ is a coproduct of contractible categories.
  \end{enumerate}
\end{proposition}
\begin{proof}
  (i) implies (ii). Given a rigid category, any automorphism $f:x\to x$ is parallel with the identity and thus has to be equal to it.
  \\
  (ii) implies (i). Given two parallel morphisms $f,g:x\to y$, we have $g^{-1}\circ f=\id_x$ and thus $f=g$.
  \\
  (i) implies (iii). Write $\D$ for category of connected components of~$\W$: this is the discrete category whose objects are the equivalence classes~$[x]$ of objects~$x$ of~$\W$ under the equivalence relation identifying $x$ and $y$ whenever there is a morphism $f:x\to y$ in~$\W$. The quotient functor $Q:\W\to\D$ is full because $\D$ is discrete, faithful because $\W$ is rigid, and surjective on objects by construction of~$\D$. It is thus an equivalence of categories.
  \\
  (iii) implies (i). Given an equivalence $F:\W\to\D$ to a discrete category~$\D$, any two parallel morphisms~$f,g:x\to y$ have the same image $Ff=Fg$ (which is an identity) because~$\D$ is discrete, and are thus equal because~$F$ is faithful.
  \\
  (iii) implies (iv). Consider a functor $F:\W\to\D$ which is an equivalence with $\D$ discrete. Given $x\in\D$, we write $F^{-1}x$ for the full subcategory of~$\W$ whose objects are sent to~$x$ by~$F$. Since~$\D$ is discrete, for any morphism $f:x\to y$ in $\D$, we have $Fx=Fy$, from which it follows that $\W\isoto\bigsqcup_{x\in D}F^{-1}x$. Since $\D$ is discrete, each $F^{-1}x$ is non-empty, connected and rigid and thus contractible by \cref{contractible-prop}.
  \\
  (iv) implies (iii). If $\W\isoto\bigsqcup_{i\in I}\W_i$ with $\W_i$ contractible, \ie $\W_i\equivto 1$, then $\W\equivto\bigsqcup_{i\in I}1$ because equivalences are closed under coproducts and thus $\W$ is equivalent to a discrete category.
\end{proof}

\noindent
The fact that $\W\subseteq\C$ is rigid can be thought of here as the fact that coherence condition (C1) holds for~$\C$, relatively to~$\W$: any two parallel structural morphisms are equal. Conditions~(iii) and (iv) can also be interpreted as stating that $\W$ is a set, up to equivalence.

General notions of quotients (with respect to a subcategory, or to a general notion of congruence both on objects and morphisms) have been developed in~\cite{bednarczyk1999generalized}, and are non-trivial to study and construct. However, when quotienting a category~$\C$ by a rigid subgroupoid~$\W$, we have the following simple description. In this case, we define the two following equivalence relations $\sim_\W$, that we often simply write~$\sim$.
\begin{itemize}
\item We write~$\sim_\W$ for the equivalence relation on objects of~$\C$ such that $x\sim y$ whenever there is a morphism $f:x\to y$ in~$\W$. When it exists, such a morphism is unique by rigidity of~$\W$ and noted~$w_{x,y}:x\to y$.
\item We also write~$\sim_\W$ for the equivalence relation on morphisms of~$\C$ such that for $f:x\to y$ and $f':x'\to y'$ we have $f\sim f'$ whenever there exists morphisms $v:x\to x'$ and $w:y\to y'$ in~$\W$ making the following diagram commute:
  \[
    \begin{tikzcd}
      x\ar[d,dashed,"v"']\ar[r,"f"]&y\ar[d,dashed,"w"]\\
      x'\ar[r,"f'"']&y'
    \end{tikzcd}
  \]
\end{itemize}

\begin{definition}
  \label{rigid-quotient-category}
  Let~$\C$ be a category equipped with a rigid subgroupoid~$\W$. We write~$\C/{\sim_{\W}}$ for the category where
  \begin{itemize}
  \item an object~$[x]$ is an equivalence class of an object~$x$ of~$\C$ under~$\sim$,
  \item a morphism $[f]:[x]\to[y]$ is the equivalence class of a morphism $f:x\to y$ in~$\C$ under~$\sim$,
  \item given morphisms $f:x\to y$ and $g:y'\to z$ with $[y]=[y']$, their composition is the morphism $[g]\circ[f]=[g\circ w_{y,y'}\circ f]$:
    \[
      \begin{tikzcd}
        x\ar[r,"f"]&y\ar[r,dashed,"w_{y,y'}"]&y'\ar[r,"g"]&z
      \end{tikzcd}
    \]
  \item the identity on an object~$[x]$ is~$[\id_x]$.
  \end{itemize}
\end{definition}

\noindent
Note that there is a canonical functor $[-]:\C\to\C/{\sim_\W}$ sending (\resp a morphism) to its equivalence class.

\begin{lemma}
  The above category is well-defined.
\end{lemma}
\begin{proof}
  We can check that the operations are well-defined and axioms of categories are satisfied.
  \begin{itemize}
  \item Composition is compatible with the equivalence relation. Given $f_1:x_1\to y_1$, $f_2:x_2\to y_2$, $g_1:y_1'\to z_1$, $g_2:y_2'\to z_2$ such that $f_1\sim f_2$ and $g_1\sim g_2$ (and thus $x_1\sim x_2$, $y_1\sim y_2$, $y_1'\sim y_2'$ and $z_1\sim z_2$) which are composable (\ie $y_1\sim y_1'$ and $y_2\sim y_2'$), the following diagram shows that $[g_1]\circ[f_1]=[g_2]\circ[f_2]$:
    \[
      \begin{tikzcd}
        x_1\ar[d,dashed,"w_{x_1,x_2}"']\ar[r,"f_1"]&y_1\ar[d,dashed,"w_{y_1,y_2}"description]\ar[r,dashed,"w_{y_1,y_1'}"]&\ar[d,dashed,"w_{y_1',y_2'}"description]y_1'\ar[r,"g_1"]&z_1\ar[d,dashed,"w_{z_1,z_2}"]\\
        x_2\ar[r,"f_2"']&y_2\ar[r,dashed,"w_{y_2,y_2'}"']&y_2'\ar[r,"g_2"']&z_2\\
      \end{tikzcd}
    \]
    where the squares on the left and right respectively commute because $f_1\sim f_2$ and $g_1\sim g_2$ and the one in the middle does by rigidity of~$\W$.
  \item Identities are compatible with the equivalence relations. Given objects~$x$ and~$y$ of~$\C$ such that $x\sim y$, the diagram
    \[
      \begin{tikzcd}
        x\ar[d,dotted,"w_{x,y}"']\ar[r,"\id_x"]&x\ar[d,dotted,"w_{x,y}"]\\
        y\ar[r,"\id_{y}"']&y
      \end{tikzcd}
    \]
    commutes showing that we have $\id_x\sim\id_y$.
  \end{itemize}
  Associativity of composition and the fact that identities are neutral element for composition follow immediately from the fact that those properties are satisfied in~$\C$.  
\end{proof}

\begin{proposition}
  \label{prop:rigid-quotient}
  The category $\C/\W$ is isomorphic to $\C/{\sim_\W}$.
\end{proposition}
\begin{proof}
  We show that the category has the universal property of \cref{quotient-category}. The quotient functor $[-]:\C\to\C/\W$ is $\W$-strict, \ie for any morphism $w:x\to y$, we have $[w]=[\id_y]$:
  \[
    \begin{tikzcd}
      x\ar[d,dashed,"w"']\ar[r,"w"]&y\ar[d,dashed,"\id_y"]\\
      y\ar[r,"\id_y"']&y
    \end{tikzcd}
  \]
  Moreover, a $\W$-strict functor $F:\C\to\D$ induces a unique functor $\tilde F:\C/\W\to\D$. Namely, by $\W$-strictness, two objects (\resp morphisms) which are equivalent have the same image by~$F$.
\end{proof}

When $\W\subseteq\C$ is not rigid, we can have a similar description of the
quotient, but the description is more complicated. Namely, if we are trying to
compose two morphisms $[f]$ and $[g]$ in the quotient with $f:x\to y$ and
$g:y'\to z$, we might have multiple morphisms $y\to y'$ in~$\W$ (say $v$ and
$w$),
\[
  \begin{tikzcd}
    x\ar[r,"f"]&y\ar[r,dashed,shift left,"v"]\ar[r,dashed,shift right,"w"']&y'\ar[r,"g"]&z
  \end{tikzcd}
\]
In such a situation, the compositions $g\circ v\circ f$ and $g\circ w\circ f$ should be identified in the quotient.
This observation suggests that the construction of the quotient category~$\C/\W$, when~$\W$ is not rigid, is better described in two steps: we first formally make~$\W$ rigid, and then apply \cref{prop:rigid-quotient}.

\begin{definition}
  A functor $F:\C\to\D$ is \emph{$\W$-rigid} when for any parallel morphisms
  $f,g:x\to y$ of~$\C$ we have $Ff=Fg$.
\end{definition}

\begin{definition}
  The \emph{$\W$-rigidification}~$\rquot\C\W$ of~$\C$ is the category equipped
  with a $\W$-rigid functor $\C\to\rquot\C\W$, such that any $\W$-rigid functor
  $F:\C\to\D$ extends uniquely as a functor $\tilde F:\rquot\C\W\to\D$ making the
  following diagram commute:
  \[
    \begin{tikzcd}
      \C\ar[d]\ar[r,"F"]&\D\\
      \rquot\C\W\ar[ur,dotted,"\tilde F"']
    \end{tikzcd}
  \]
\end{definition}

\begin{lemma}
  \label{lem:rigidification}
  The category~$\rquot\C\W$ is the category obtained from~$\C$ by quotienting
  morphisms under the smallest congruence (\wrt composition) identifying any two
  parallel morphisms of~$\W$.
\end{lemma}

\begin{proposition}
  \label{prop:quotient-2steps}
  The quotient $\C/\W$ is isomorphic to $(\rquot\C\W)/\tilde \W$ where
  $\tilde \W$ is the set of equivalence classes of morphisms in~$\W$ under the
  equivalence relation of \cref{lem:rigidification}.
\end{proposition}
\begin{proof}
  Since any $\W$-strict functors are $\W$-rigid, we have that any $\W$-strict
  functor extends as a unique functor $\tilde F:\rquot\C\W\to\D$ which is
  $\W$-rigid, and thus as a unique functor $\C/\tilde\W\to\D$:
  \[
    \begin{tikzcd}[baseline={([yshift=5pt]current bounding box.south)}]
      \C\ar[d]\ar[r,"F"]&\D\\
      \rquot\C\W\ar[d]\ar[ur,dotted,"\tilde F"']\\
      \C/\tilde\W\ar[uur,bend right,dotted,"\tilde{\tilde F}"']
    \end{tikzcd} \qedhere
  \]
\end{proof}

\noindent
A consequence of the preceding explicit description of the quotient is the
following:

\begin{lemma}
  \label{lem:quotient-surj}
  The quotient functor $\C\to\C/\W$ is surjective on objects and full.
\end{lemma}
\begin{proof}
  By \cref{prop:quotient-2steps}, the quotient functor is the composite of the
  quotient functors $\C\to\rquot\C\W\to\C/\W$. The first one is surjective on
  objects and full by \cref{lem:rigidification} and the second one is surjective
  on objects and full by \cref{prop:rigid-quotient}.
\end{proof}

\noindent
This entails the following theorem, which is the main result of the section. Its meaning can be explained by taking the point of view given above: thinking of the category~$\C$ as describing a structure, and of~$\W$ as the part of the structure we want to strictify, the structure is equivalent to its strict variant if and only if the quotiented structure does not itself bear non-trivial geometry, in the sense of \cref{lem:rigid}.

\begin{theorem}
  \label{prop:rigid-equivalence}
  \label{rigid-equivalence}
  Suppose that $\W$ is a subgroupoid of~$\C$. The quotient functor
  $[-]:\C\to\C/\W$ is an equivalence of categories if and only if~$\W$ is rigid.
\end{theorem}
\begin{proof}
  Since the quotient functor is always surjective and full by
  \cref{lem:quotient-surj}, it remains to show that it is faithful if and only
  if~$\W$ is rigid.
  Suppose that the quotient functor is faithful. Given $w,w':x\to y$ in~$\W$, by
  \cref{lem:rigidification,prop:quotient-2steps} we have $[w]=[w']$ and thus
  $w=w'$ by faithfulness.
  Suppose that~$\W$ is rigid. The category~$\C/\W$ then admits the description
  given in \cref{prop:rigid-quotient}. Given $f,g:x\to y$ in $\C$ such that
  $[f]=[g]$, there is $v:x\to x$ and $w:y\to y$ such that $w\circ f=g\circ
  v$. By rigidity, both $v$ and $w$ are identities and thus $f=g$.
\end{proof}

\begin{example}
  \label{ex:rigid-equivalence}
  As a simple example, consider the groupoid~$\C$ freely generated by the graph
  \[
    \begin{tikzcd}
      x\ar[r,shift left,"f"]\ar[r,shift right,"g"']&y 
    \end{tikzcd}
  \]
  The subgroupoid generated by $W=\set{g}$ is rigid, so that~$\C$ is equivalent to the quotient category $\C/W$, which is the groupoid generated by the graph
  \[
    \begin{tikzcd}
      x\ar[loop right,"f"]
    \end{tikzcd}
  \]
  However, the groupoid generated by $W=\set{f,g}$ is not rigid (since we do not have $f=g$). In this case, $\C$ is not equivalent to the quotient category $\C/W$, which is the terminal category.
\end{example}


\begin{remark}
  By taking~$\C$ to be~$\W$ in \cref{rigid-equivalence}, we obtain that a
  category $\W$ is rigid if and only if it is equivalent to the discrete
  category of its connected components. This thus provides an alternative proof
  of condition (iii) of \cref{rigid-characterizations}.
\end{remark}


\subsection{Coherence for algebras}
%
%

Given two categories $\C$ and $\D$, an \emph{algebra} of~$\C$ in~$\D$ is a
functor from $\C$ to $\D$. In the following, we will mostly be interested in the
case where~$\D=\Cat$: if we think of the category~$\C$ as describing an
algebraic structure (\eg the one of monoidal categories), an algebra can be
thought of as a category actually possessing this structure (an actual monoidal
category).
We write $\Alg(\C,\D)$ for the category whose objects are algebras and morphisms
are natural transformations, and $\Alg(\C)=\Alg(\C,\Cat)$. Note that any functor
\[
  F:\C\to\C'
\]
induces, by precomposition, a functor
\[
  \Alg(F,\D):\Alg(\C',\D)\to\Alg(\C,\D)
  \text.
\]
We can characterize situations where two categories give rise to the same algebras as follows.

\begin{proposition}
  \label{prop:equiv-alg}
  A functor $F:\C\to\C'$ is an equivalence if and only if there is a family of equivalences of categories $\Alg(\C,\D)\equivto\Alg(\C',\D)$ which is natural in~$\D$.
  %
  %
\end{proposition}
\begin{proof}
  \newcommand{\K}{\mathcal{K}}
  Given a $2$-category~$\K$, one can define a Yoneda functor
  \begin{align*}
    Y_\K:\K^\op&\to[\K,\Cat]\\
    c&\mapsto\K(c,-)
  \end{align*}
  where $\Cat$ is the $2$-category of categories, functors and natural
  transformations, and~$[\K,\Cat]$ denotes the 2-category of 2-functors
  $\K\to\Cat$, transformations and modifications. In particular, given $0$-cells
  $c\in\K^\op$ and $d\in\K$, we have $Y_\K cd=\K(c,d)$. The Yoneda lemma
  states that this functor is a local isomorphism (this is a particular case of
  the Yoneda lemma for bicategories~\cite[Corollary 8.3.13]{johnson20212}): this
  means that, for every objects $c,d\in\K^\op$, we have an isomorphism of
  categories
  \[
    [\K,\Cat](Y_\K c,Y_\K d)
    \isoto
    \K(d,c)
    \text.
  \]
  In particular, taking $\K=\Cat$ (and ignoring size issues, see below for a way
  to properly handle this), the Yoneda functor sends a category $\C\in\Cat^\op$
  to $Y_\Cat\C$, \ie $\Cat(\C,-)$, \ie $\Alg(\C,-)$.
  Given category~$\C$ and $\C'$, by the Yoneda lemma, we thus have an
  isomorphism of categories
  \[
    [\Cat,\Cat](\Alg(\C,-),\Alg(\C',-))
    \isoto
    \Cat(\C',\C)
  \]
  which is compatible with $0$-composition in~$\Cat$. The categories~$\C$ and
  $\C'$ are thus equivalent, if and only if the categories $\Alg(\C,-)$ and
  $\Alg(\C',-)$ are equivalent, which is the case if and only if there is a
  family of equivalences of categories between $\Alg(\C,\D)$ and $\Alg(\C,\D)$
  natural in~$\D$.
  %
  %
\end{proof}

\begin{proof}[Alternative proof of \cref{prop:equiv-alg}]
  We provide here an alternative more pedestrian proof, inspired of~\cite[Lemma~5.3.1]{gaussent2015coherent}, which does not require ignoring size issues.
  Suppose that $F$ is an equivalence of categories, with pseudo-inverse
  $G:\C'\to\C$,
  \ie we have $G\circ F\isoto\Id_\C$ and $F\circ G\isoto\Id_{\C'}$.
  We define functors
  \begin{align*}
    \Alg(F,\D):\Alg(\C',\D)&\to\Alg(\C,\D)&
    \Alg(G,\D):\Alg(\C,\D)&\to\Alg(\C',\D)
    \\
    A&\mapsto A\circ F
    &
    A&\mapsto A\circ G
  \end{align*}
  Those induce an equivalence between $\Alg(\C,\D)$ and $\Alg(\C',\D)$ since,
  for $A\in\Alg(\C,\D)$ and $A'\in\Alg(\C',\D)$, we have
  \begin{align*}
    \Alg(G,\D)\circ\Alg(F,\D)&=A'\circ F\circ G\simeq A'
    \\
    \Alg(F,\D)\circ\Alg(G,\D)&=A\circ G\circ F\simeq A
  \end{align*}
  Moreover, the family of equivalences $\Alg(F,\D)$ is natural in~$\D$, and
  similarly for $\Alg(G,\D)$. Namely, given $H:\D\to\D'$ and considering the
  functor
  \begin{align*}
    \Alg(\C,H):\Alg(\C,\D)&\to\Alg(\C,\D')\\
    A&\mapsto H\circ A
  \end{align*}
  as well as the variant with $\C'$ instead of~$\C$, the diagram
  \[
    \begin{tikzcd}[column sep=large]
      \Alg(\C',\D)\ar[d,"{\Alg(F,\D)}"']\ar[r,"{\Alg(\C',H)}"]&\Alg(\C',\D')\ar[d,"{\Alg(F,\D')}"]\\
      \Alg(\C,\D)\ar[r,"{\Alg(\C,H)}"']&\Alg(\C,\D')
    \end{tikzcd}
  \]
  commutes: an object $A\in\Alg(\C',\D)$ is sent to $H\circ A\circ F$ by both
  sides.

  Conversely, suppose given an equivalence of categories
  \[
    \Phi_\D:\Alg(\C',\D)\leftrightarrow\Alg(\C,\D):\Psi_\D
  \]
  which is natural in~$\D$. We define
  \begin{align*}
    F=\Phi_{\C'}(\Id_{\C'}):\C&\to\C'
    &
    G=\Psi_{\C}(\Id_\C):\C'&\to\C
  \end{align*}
  and we have
  \[
    G\circ F
    =
    G\circ\Phi_{\C'}(\Id_{\C'})
    =
    \Phi_{\C'}(G\circ\Id_{\C'})
    =
    \Phi_{\C'}(\Psi_{\C}(\Id_\C))
    \isoto
    \Id_{\C'}
  \]
  (the second equality is naturality), and similarly for $G\circ F\isoto\Id_\C$.
\end{proof}

\begin{remark}
  We would like to point out a subtle point with respect to naturality in the above theorem. Given a functor $F:\C\to\C'$, it is always the case that the induced functors~$\Alg(F,\D)$ form a family which is natural in~$\D$. Suppose moreover that all the functors~$\Alg(F,\D)$ are equivalences. We do not see any argument to show that the pseudo-inverse functors form a natural family, which is why we have to additionally impose this condition.
\end{remark}

\noindent
As a particular application, we have the following proposition, which can be interpreted as the equivalence of coherence conditions (C1) and a strengthened variant of (C4):

\begin{proposition}
  \label{prop:alg-str}
  Let~$\C$ be a category together with a subgroupoid~$\W$. We have, for any category~$\D$, a functor
  \begin{equation}
    \label{quotient-precomposition}
    F_\D:\Alg(\C/\W,\D)\to\Alg(\C,\D)
  \end{equation}
  induced by precomposition with the quotient functor $\C\to\C/\W$. These functors $(F_\D)_{\D\in\Cat}$ form a family of equivalences of categories, natural in~$\D$, if and only if~$\W$ is rigid.
\end{proposition}
\begin{proof}
  By \cref{prop:rigid-equivalence}, $\W$ is rigid if and only if the quotient
  functor $\C\to\C/\W$ is an equivalence, and we conclude by
  \cref{prop:equiv-alg}.
\end{proof}


\begin{remark}
  \TODO{comprendre pourquoi ça fonctionne dans le cas des théories de Lawvere
    (cf le papier de Hyland et Power)}
  It can be wondered whether the case where $\D=\Cat$ is enough, \ie whether the
  quotient functor $\C\to\C/\W$ is an equivalence whenever the functor
  $\Alg(\C/W)\to\Alg(\C)$ it induces is an equivalence.
  We leave it as an open question, but remark here that it cannot follow from
  general results: it is not the case that two categories $\C$ and~$\C'$ are
  equivalent whenever $\Alg(\C)$ and $\Alg(\C')$ are equivalent. It is namely
  known that two Cauchy equivalent categories give rise to the same algebras
  in~$\Cat$, see~\cite{lack2022flat}, so that the categories
  \begin{align*}
    \C
    &=
    \begin{tikzcd}
      x\ar[loop right,"e"]
    \end{tikzcd}
    &
    \C'
    &=
    \begin{tikzcd}[ampersand replacement=\&]
      x\ar[r,"f"]\&y
    \end{tikzcd}
  \end{align*}
  where~$e\circ e=e$ have the same algebras.
\end{remark}









\section{Coherent abstract rewriting systems}
\label{sec:ars}
\label{sec:coh-ars}
Previous sections illustrate the importance of the property of being rigid for a groupoid, and we now provide tools to show this in practice, based on tools originating from rewriting theory. In the same way the theory of rewriting can be studied ``abstractly''~\cite{baader1999term,huet1980confluent,bezem2003term}, \ie without taking in consideration the structure of the objects being rewritten, we first develop the coherence theorems of interest in an abstract setting.

\subsection{Extended abstract rewriting systems}
The categorical formalization of the notion of rewriting system given here is based on the notion of polygraph~\cite{polygraphs,burroni1993higher,street1976limits}.

\begin{definition}
  An \emph{abstract rewriting system}, or \ARS{}, or \emph{$1$-polygraph} is a
  diagram
  \[
    \begin{tikzcd}
      \P_0
      &
      \ar[l,"s_0"',shift right]
      \ar[l,"t_0",shift left]
      \P_1
    \end{tikzcd}
  \]
  in the category~$\Set$.
\end{definition}

\noindent
An \ARS{} is simply another name for a directed graph. It consists of a set
$\P_0$ whose elements are the \emph{objects} of interest, a set $\P_1$ of
\emph{rewriting rules} and two functions $s_0$ and $t_0$ respectively
associating to a rewriting rule its \emph{source} and \emph{target}. We often
write
\[
  a:x\to y
\]
to denote a rewriting rule~$a$ with $s_0(a)=x$ and $t_0(a)=y$.
We write $\freecat\P_1$ for the set of \emph{rewriting paths} in the \ARS{}: its
elements are (possibly empty) finite sequences~$a_1,\ldots,a_n$ of rewriting
steps, which are composable in the sense that $t_0(a_i)=s_0(a_{i+1})$ for
$1\leq i<n$. Writing $x_i=t_0(a_i)$, such a path can thus be represented as
\[
  \begin{tikzcd}
    x_0\ar[r,"a_1"]&
    x_1\ar[r,"a_2"]&
    \cdots
    \ar[r,"a_n"]
    &x_n
    \text.
  \end{tikzcd}
\]
The source and target of such a rewriting path are respectively $x_0=s_0(a_1)$ and $x_n=t_0(a_n)$, and $n$ is called the \emph{length} of the path. We sometimes write
\[
  p:x\pathto y
\]
to indicate that~$p$ is a rewriting path with $x$ as source and~$y$ as target. Given two composable paths $p:x\pathto y$ and $q:y\pathto z$, we write $p\pcomp q:x\pathto z$ for their concatenation.


A \emph{morphism} $f:\P\to\Q$ of \ARS{} is a pair of functions $f_0:\P_0\to\Q_0$
and $f_1:\P_1\to\Q_1$ such that $s_0\circ f_1=f_0\circ s_0$ and
$t_0\circ f_1=f_0\circ t_0$:
\[
  \begin{tikzcd}
    \P_0
    \ar[d,"f_0"']
    &
    \ar[l,"s_0"',shift right]
    \ar[l,"t_0",shift left]
    \P_1
    \ar[d,"f_1"]
    \\
    \Q_0
    &
    \ar[l,"s_0"',shift right]
    \ar[l,"t_0",shift left]
    \Q_1
  \end{tikzcd}
\]
We write~$\nPol1$ for the resulting category or \ARS{} and their
morphisms. There is a forgetful functor $\Cat\to\nPol1$, sending a category~$C$
to the \ARS{} whose objects are those of~$C$ and whose rewriting steps are the
morphisms of~$C$.

\begin{lemma}
  The forgetful functor $\Cat\to\nPol1$ admits a left adjoint
  $\freecat-:\nPol1\to\Cat$. It sends an \ARS{} to the category with~$\P_0$ as
  objects and~$\freecat\P_1$ as morphisms, where composition is given by
  concatenation of paths and identities are the empty paths
\end{lemma}
\begin{proof}
  This fact is easily checked directly, but an abstract argument for the existence of the left adjoint in such situations is the following: the categories $\Cat$ and $\nPol1$ are models of projectives sketches and the forgetful functor $\Cat\to\nPol1$ is induced by a functor of sketches (the ``inclusion'' of the sketch of \ARS{} into the sketch of categories) and, as such, it admits a left adjoint~\cite[Theorem~4.1]{barr2000toposes}, \cite[Proposition~15.1.3]{polygraphs}.
\end{proof}

\noindent
As a variant of the preceding situation, we can consider the forgetful functor $\Gpd\to\nPol1$, from the category of groupoids. It also admits a left adjoint $\freegpd-:\nPol1\to\Gpd$, which can be described as follows. Given an \ARS{}~$\P$, we write $\P^\pm$ for the \ARS{}
\[
  \begin{tikzcd}
    \P_0
    &
    \ar[l,"s_0"',shift right]
    \ar[l,"t_0",shift left]
    (\P_1\times\set{-,+})
  \end{tikzcd}
\]
Its objects are the same as for~$\P$. A rule in~$\P^\pm_1$ is a pair $(a,\epsilon)$ consisting of a rewriting rule $a\in\P_1$ and $\epsilon\in\set{-,+}$, which will be noted~$a^\epsilon$ in the following. The source and target
maps are given by
\begin{align*}
  s_0(a^+)&=s_0(a)
  &
  t_0(a^+)&=t_0(a)
  &
  s_0(a^-)&=t_0(a)
  &
  t_0(a^-)&=s_0(a)
\end{align*}
We can think of $a^+$ as corresponding to $a$ and $a^-$ as corresponding to $a$
taken ``backward''. A \emph{rewriting zig-zag} is a path
$a_1^{\epsilon_1},\ldots,a_n^{\epsilon_n}$ in $\P^\pm$. The intuition is that a
zig-zag is a ``non-directed'' rewriting path, consisting of rewriting steps,
some of which are taken backward. We write
\[
  p:x\zzto y
\]
to indicate that $p$ is a zig-zag from~$x$ to~$y$.
Two zig-zags are \emph{congruent} when they are related by the smallest
congruence~$\sim$ such that, for every rewriting rule $a:x\to y$, we have
\begin{align}
  \label{zig-zag-congruence}
  a^+a^-&\sim\id_x
  &
  a^-a^+&\sim\id_y
\end{align}
We write $\freegpd\P_1$ for the set of zig-zags up to congruence.

\begin{lemma}
  The category~$\freegpd\P$ with $\P_0$ as objects, $\freegpd\P_1$ as morphisms,
  where composition is given by concatenation of paths up to congruence, is the
  free groupoid on~$\P$.
\end{lemma}

\noindent
We have a canonical function $i_1:\P_1\to\freegpd\P_1$, sending a rewriting step $a$ to $a^+$. Writing $\freegpd s_0,\freegpd t_0:\freegpd\P_1\to\P_0$ for the source and target maps, it induces a morphism of \ARS{} by taking the identity on objects:
\[
  \begin{tikzcd}
    \P_0
    \ar[d,"\id"']
    &
    \ar[l,"s_0"',shift right]
    \ar[l,"t_0",shift left]
    \P_1
    \ar[d,"i_1"]
    \\
    \P_0
    &
    \ar[l,"\freegpd s_0"',shift right]
    \ar[l,"\freegpd t_0",shift left]
    \freegpd\P_1
  \end{tikzcd}
\]
Writing $i:\P\to\freegpd\P$ for this morphism of \ARS{}, the universal property
of $\P$ states that any morphism of \ARS{} $F:\P\to\C$, where $\C$ is a groupoid,
extends uniquely as a functor $\tilde F:\freegpd\P\to\C$ making the following
diagram commute:
\[
  \begin{tikzcd}
    \P\ar[d,"i"']\ar[r,"F"]&\C\\
    \freegpd\P\ar[ur,dotted,"\tilde F"']
  \end{tikzcd}
\]

In the following, in order to avoid working with equivalence classes when working with elements of $\freegpd\P_1$, we will instead only implicitly consider zig-zags $a_1^{\epsilon_1},\ldots,a_n^{\epsilon_n}$ which are \emph{reduced}, in the sense that they satisfy the following property: for every index $i$ with $1\leq i<n$, we have that $a_i=a_{i+1}$ implies $\epsilon_i=\epsilon_{i+1}$. This is justified by the following result.

\begin{lemma}
  \label{zz-reduce}
  The equivalence class under $\sim$ of a zig-zag contains a unique reduced
  zig-zag.
\end{lemma}
\begin{proof}
  Consider the string rewriting system on words over $\P^\pm_1$ with rules
  $a^+a^-\To\id$ and $a^-a^+\To\id$ for $a\in\P_1$, corresponding to
  \cref{zig-zag-congruence}. It is length-reducing and thus terminating. Its
  critical pairs (whose sources are $a^-a^+a^-$ and $a^+a^-a^+$) are
  confluent, it is thus confluent. We deduce that any equivalence class contains
  a unique normal form, and those are precisely reduced zig-zags.
\end{proof}

\noindent
Given a path $p:x\pathto y$, we write $p^+:x\zzto y$ (\resp $p^-:y\zzto x$) for the zig-zag obtained by adding a ``$+$'' (\resp ``$-$'') exponent to every step of the rewriting path. In particular, the first operation induces a canonical inclusion $\freecat i_1:\freecat\P_1\to\freegpd\P_1$, defined by $\freecat i_1(p)=p^+$, witnessing for the fact that rewriting paths are particular zig-zags. We will sometimes leave its use implicit in the following. Note that \cref{zz-reduce} implies that $\freecat i_1$ is injective.

We think here of \ARS{} as abstractly describing some algebraic structures. It is thus natural to extend this notion in order to take in account the coherence laws that these structures should possess. This can be done as follows.

\begin{definition}
  An \emph{extended abstract rewriting system}, or \ARS{2}, or
  \emph{$2$-polygraph}, $\P$ consists of an \ARS{} as above, together with a
  set~$\P_2$ and two functions $s_1,t_1:\P_2\to\freegpd{\P_1}$, such that
  \begin{align*}
    \freegpd s_0\circ s_1&=\freegpd s_0\circ t_1
    &
    \freegpd t_0\circ s_1&=\freegpd t_0\circ t_1  
  \end{align*}
  This data can be summarized as a diagram
  \[
    \begin{tikzcd}[column sep=large]
      &\ar[dl,shift right,"s_0"',near start]\ar[dl,shift left,"t_0",near start]\ar[d,"i_1"]\P_1&\ar[dl,shift right,"s_1"']\ar[dl,shift left,"t_1"]\P_2\\
      \P_0&\ar[l,shift right,"\freegpd s_0"']\ar[l,shift left,"\freegpd t_0"]\freegpd\P_1
    \end{tikzcd}
  \]
\end{definition}

\noindent
In a \ARS2, the elements of~$\P_2$ are \emph{coherence relations} and the
functions $s_1$ and $t_1$ respectively describe their source and target, which
are rewriting zig-zags. We sometimes write
\[
  A:p\To q
\]
to indicate that $A\in\P_2$ is a coherence relation which admits $p$ (\resp $q$)
as source (\resp target), which can be thought of as a $2$-cell
\[
\begin{tikzcd}[column sep=large]
  x\ar[r,bend left,"p"]\ar[r,bend right,"q"']\ar[r,phantom,"A\!\Downarrow\!\phantom{A}"]&y
\end{tikzcd}
\]
where $x$ (\resp $y$) is the common source (\resp target) of~$p$ and~$q$.
%

\begin{definition}
  The \emph{groupoid presented} by a \ARS{2}~$\P$, denoted by $\pgpd\P$, is the
  groupoid obtained from the free groupoid generated by the underlying \ARS{} by
  quotienting morphisms under the smallest congruence identifying the source and
  the target of any element of~$\P_2$.
\end{definition}

\noindent
More explicitly, the groupoid $\pgpd\P$ thus has~$\P_0$ as set of objects, the
set~$\freegpd\P_1$ of rewriting zig-zags as morphisms, quotiented by the
smallest equivalence relation~$\cohto$ such that
\[
  p\pcomp q\pcomp r\cohto p\pcomp q'\pcomp r
\]
for every rewriting zig-zags $p$ and $r$ and coherence relation~$A:q\To q'$,
which are suitably composable:
\begin{equation}
  \label{coh-step}
  \begin{tikzcd}[column sep=large]
    x'\ar[r,"p"]&x\ar[r,bend left,"q"{name=q}]\ar[r,bend right,"q'"'{name=q'}]&y\ar[r,"r"]&y'
    \ar[from=q,to=q',Rightarrow,"A"',shorten=4pt]
  \end{tikzcd}
\end{equation}
We write $\Leftrightarrow$ for the smallest symmetric relation identifying path
$p\pcomp q\pcomp r$ and $p\pcomp q'\pcomp r$ when there is a coherence relation
$A:q\To q'$ as pictured above, so that $\cohto$ is the reflexive transitive
closure of $\Leftrightarrow$.
Given a rewriting zig-zag $p\in\freegpd\P_1$, we write $\pgpd p$ for the
corresponding morphism in~$\pgpd\P$, \ie its equivalence class
under~$\cohto$. Given a zig-zag $p:x\to y$ in $\freegpd P$, we write
$\pgpd p:x\to y$ for its equivalence class.

\begin{remark}
  \label{free-2-groupoid}
  A more categorical approach to the equivalences between zig-zags can be
  developed as follows, see \cite{polygraphs} for details. A \emph{$2$-groupoid}
  is a $2$-category whose $1$- and $2$-cells are invertible. A \ARS2 freely
  generates a $2$-groupoid, whose underlying $1$-groupoid is the one freely
  generated by the underlying \ARS1 of~$\P$, and containing the coherence
  relations as $2$-cells. Given zig-zags $p,q:x\to y$, we then have $p\cohto q$
  if and only if there is a $2$-cell $p\To q$ in the free $2$-groupoid: the
  $2$-cells can thus be thought of as witnesses for the equivalences of
  zig-zags. We do not further detail this approach here, since it is not
  required, but it would be for instance needed if we were interested in higher
  coherence laws.
\end{remark}

There are many \ARS2 presenting a given groupoid. In particular, one can always
perform the following transformations on \ARS2, while preserving the presented
groupoid. Those are analogous to the transformations that one can perform on
group presentations (while preserving the presented group) first studied by
Tietze~\cite{tietze1908topologischen,lyndon1977combinatorial}.

\begin{definition}
  \label{ars2-tietze}
  The \emph{Tietze transformations} are the following possible transformations
  on a \ARS2~$\P$:
  \begin{enumerate}[(T2)]
  \item[(T1)] given a zig-zag $p:x\zzto y$, add a new rewriting rule $a:x\to y$
    in $\P_1$ together with a new coherence relation $A:a\To p$ in~$\P_2$,
  \item[(T2)] given zig-zags $p,q:x\zzto y$ such that $p\cohto q$, add a new
    coherence relation $A:p\To q$ in~$\P_2$.
  \end{enumerate}
  The \emph{Tietze equivalence} is the smallest equivalence relation on \ARS2
  identifying $\P$ and $\Q$ whenever $\Q$ can be obtained from~$\P$ by a Tietze
  transformation (T1) or (T2).
\end{definition}

\noindent
It is easy to see that the Tietze transformations are ``correct'', in the sense
that they preserve the presented groupoid. With more
work~\cite[Chapter~5]{polygraphs}, it is even possible to show that those
transformations are ``complete'', in the sense that any two \ARS2 presenting the
same groupoid are Tietze equivalent. We only state the first direction here
since this is the only one we will need:

\begin{proposition}
  \label{tietze-correct}
  Any two Tietze equivalent \ARS2 present isomorphic groupoids.
\end{proposition}

\noindent
From the transformations of \cref{ars2-tietze}, it is possible to derive other transformations, which also preserve the presented groupoid.

\begin{lemma}
  \label{tietze-remove-rule}
  Suppose that~$\P$ is a \ARS2 containing a rewriting rule $a:x\to y$ and a relation $A:p\To q$ such that $a$ occurs exactly once in the source~$p$, \ie $p=p_1\pcomp a\pcomp p_2$, and does not occur in the target:
  \[
    \begin{tikzcd}
      &x'\ar[r,"a",""{name=a,below}]&y'\ar[dr,bend left,"p_2"]\\
      x\ar[ur,bend left,"p_1"]\ar[rrr,"q"',""{name=q,above}]&&&y\\
      \ar[from=a,to=q,Rightarrow,"A"',shorten=5pt]
    \end{tikzcd}
  \]
  Then~$\P$ is Tietze equivalent to the \ARS2 where
  \begin{itemize}
  \item we have removed the rewriting rule~$a$,
  \item we have removed the coherence relation~$A$,
  \item we have replaced every occurrence of~$a$ in the source or target of a coherence relation by~$p_1^-\pcomp q\pcomp p_2^-$.
  \end{itemize}
\end{lemma}

\subsection{Rewriting properties}
\label{ars-rewriting}
Let~$\P$ be a \ARS{2}. For simplicity, we suppose that for every coherence relation $A:p\To q$ in~$\P_2$, we have that $p$ and $q$ are rewriting paths (as opposed to zig-zags).
We also suppose fixed a set $W\subseteq\P_1$.  We can think of $W$ as inducing a rewriting subsystem~$\sfW$ of~$\P$, with $\P_0$ as objects, $W$ as rewriting steps and
\[
  \sfW_2=\setof{A\in\P_2}{s_1(A)\in\freecat W\text{ and }t_1(A)\in\freecat W}
\]
as coherence relations, and formulate the various traditional rewriting concepts
with respect to it.
In such a situation, consider the presented groupoid $\C=\pgpd\P$. The set~$W$
of rewriting rules, induces a set of morphisms of $\C$, namely
$\setof{\pgpd w}{w\in\ W}$ that we still write $W$, which generates a
subgroupoid $\W$ of~$\C$. Our aim here is to provide rewriting tools in order to
show that $\W$ is rigid, so that $\C$ is equivalent to the quotient $\C/\W$ by
\cref{rigid-equivalence}, and moreover provide a concrete description of the
quotient category.
 
\begin{definition}
  The \ARS{2} $\P$ is \emph{$W$-terminating} if there is no infinite sequence $a_1,a_2,\ldots$ of elements of~$W$ such that every finite prefix is a rewriting path in~$\freecat{W}$.
\end{definition}

\begin{definition}
  An element $x\in\P_0$ is a \emph{$W$-normal form} when there is no rewriting
  step in~$W$ with $x$ as source.
  We say that~$\P$ is \emph{weakly $W$-normalizing} when for every $x\in\P_0$
  there exists a normal form $\nf x$ and a rewriting path $x\pathto\nf x$. In
  this case, we write $n_x:x\pathto\nf x$ for an arbitrary choice of such a
  path, which is however supposed to be the identity when $x$ is a normal form.
\end{definition}

\begin{lemma}
  \label{lem:term-wn}
  If $\P$ is $W$-terminating then it is weakly $W$-normalizing.
\end{lemma}
\begin{proof}
  Consider a maximal rewriting path $a_1,a_2,\ldots$ in $\freecat W$ starting
  from~$x$. Because $\P$ is $W$-terminating, this path is necessarily finite,
  and its target is a normal form by maximality.
\end{proof}

\begin{definition}
  A \emph{$W$-branching} is a pair of rewriting paths
  \begin{align*}
    p_1:x&\pathto y_1
    &
    q_2:x&\pathto y_2
  \end{align*}
  in~$\freecat W$ which are coinitial, \ie have the same source~$x$, which is called the \emph{source} of the branching.
  A $W$-branching is \emph{local} when both $p_1$ and $p_2$ are rewriting steps.
  A $W$-branching as above is \emph{confluent} when there is a pair of cofinal
  (with the same target) rewriting paths $q_1:y_1\pathto z$ and
  $q_2:y_2\pathto z$ in~$\freecat W$ such that
  $p_1\pcomp q_1\cohto p_2\pcomp q_2$:
  \[
    \begin{tikzcd}[sep=small]
      &\ar[dl,"p_1"']x\ar[dr,"p_2"]&\\
      y_1\ar[dr,dotted,"q_1"']\ar[rr,Leftrightarrow,"*",shorten=12pt]&&\ar[dl,dotted,"q_2"]y_2\\
      &z&
    \end{tikzcd}
  \]
\end{definition}

\noindent
Note that, in the above definitions, not only we require that we can close a
span of rewriting steps by a cospan of rewriting paths (as in the traditional
definition of confluence), but also that the confluence square can be filled
coherence relations.



\begin{definition}
  The \ARS{} $\P$ is \emph{locally $W$-confluent} when $W$-branching is
  confluent.
  It is \emph{$W$-confluent} when for every $p_1:x\pathto y_1$ and
  $p_2:x\pathto y_2$ in $\freecat W$, there exist $q_1:y_1\pathto z$ and
  $q_2:y_2\pathto z$ in $\freecat W$ such that
  $p_1\pcomp q_1\cohto p_2\pcomp q_2$. We say that~$\P$ is \emph{$W$-convergent}
  when it is both $W$-terminating and $W$-confluent.
\end{definition}

The celebrated Newman's lemma~\cite{newman42} (also sometimes called the diamond lemma) along with its traditional proof~\cite[Theorem 1.2.1 (ii)]{bezem2003term} easily generalizes to our setting:

\begin{proposition}
  \label{prop:lc+n-c}
  \label{newman}
  If $\P$ is $W$-terminating and locally $W$-confluent then it is $W$-confluent.
\end{proposition}
\begin{proof}
  The relation on objects defined by $x\geq y$ whenever there exists a rewriting
  path $p:x\pathto y$ in $\freecat W$ is a well-founded partial order
  because~$\P$ is $W$-terminating.
  We say that~$\P$ is \emph{$W$-confluent at~$x$} when every $W$-branching with
  source~$x$ confluent. We are going to show that $\P$ is locally $W$-confluent
  at $x$ for every object $x$, by well-founded induction on~$x$. In the base
  case, $x$ is a $W$-normal form and the result is immediate. Otherwise,
  consider a $W$-branching consisting of paths $a_1\pcomp p_1$ and $a_2\pcomp p_2$
  for some rewriting steps $a_1$ and $a_2$ and rewriting paths $p_1$ and $p_2$
  (we suppose that the paths are non-empty, otherwise the result is
  immediate). The following diagram shows the $W$-confluence at~$x$:
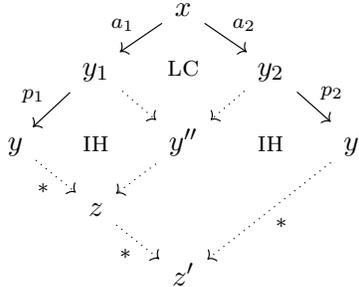
  \[
    \begin{tikzcd}[sep=small]
      &&\ar[dl,"a_1"']x\ar[dr,"a_2"]\\
      &y_1\ar[dl,"p_1"']\ar[dr,dotted]\ar[rr,phantom,"\textsc{lc}"]&&\ar[dl,dotted]y_2\ar[dr,"p_2"]\\
      y\ar[dr,dotted,"*"']\ar[rr,phantom,"\textsc{ih}"]&&\ar[dl,dotted]y''\ar[rr,phantom,"\textsc{ih}"]&&\ar[ddll,dotted,"*"]y'\\
      &z\ar[dr,dotted,"*"']\\
      &&z'
    \end{tikzcd}
  \]
  Above, the diagram \textsc{lc} is $W$-confluent by local confluence, and the
  two diagrams \textsc{ih} are by induction hypothesis.
\end{proof}

\begin{definition}
  The \ARS2 $\P$ is \emph{$W$-coherent} if for any parallel zig-zags
  $p,q:x\zzto y$ in $\freegpd W$, we have~$p\cohto q$.
\end{definition}

\noindent
The following is immediate:

\begin{lemma}
  \label{coherent-rigid}
  A \ARS2 $\P$ is $W$-coherent precisely when $\pgpd\sfW$ is a rigid subgroupoid of $\pgpd\P$.
\end{lemma}

\noindent
The traditional Church-Rosser property~\cite[Theorem~1.2.2]{bezem2003term}
generalizes as follows in our setting:

\begin{proposition}
  \label{prop:ars-zz-confl}
  If $\P$ is weakly $W$-normalizing and $W$-confluent then for any zig-zag
  $p:x\zzto y$ in $\freegpd W$, we have $\hat x=\hat y$ and
  $p\pcomp n_y\cohto n_x$, \ie the diagram
  \[
    \begin{tikzcd}
      x\ar[d,"n_x"']\ar[r,"p"]&y\ar[d,"n_y"]\\
      \hat{x}\ar[r,equals]&\hat{y}
    \end{tikzcd}
  \]
  commutes in $\pgpd\P$.
\end{proposition}
\begin{proof}
  By confluence, given a rewriting path $p:x\pathto y$ in $\freecat W$, we have
  $\nf x=\nf y$ and $p\pcomp n_y\cohto n_x$, and thus $p^+\pcomp n_y\cohto n_x$
  and $n_x\pcomp p^-\cohto n_y$, \ie the following diagrams commute in
  $\pgpd\P$:
  \[
    \begin{tikzcd}[sep=small]
      x\ar[d,"n_x"']\ar[r,"p^+"]&y\ar[d,"n_y"]\\
      \nf{x}\ar[r,equals]&\nf{y}
    \end{tikzcd}
    \qquad\qquad\qquad\qquad
    \begin{tikzcd}[sep=small]
      x\ar[d,"n_x"']&\ar[l,"p^-"']y\ar[d,"n_y"]\\
      \nf{x}\ar[r,equals]&\nf{y}
    \end{tikzcd}
  \]
  Any zig-zag $p:x\zzto y$ in $\freegpd W$ decomposes as
  $p=p_1^-q_1^+p_2^-p_2^+\ldots p_n^-p_n^+$ for some $n\in\N$ and paths $p_i$
  and $q_i$ in $\freecat W$. We thus have $p\pcomp n_y\cohto n_x$, since all the
  squares of the following diagram commute in~$\pgpd W$ by the preceding remark:
  \[
    \begin{tikzcd}
      x\ar[d,"n_x"']\ar[r,"p_1^-"]&y_1\ar[d,"n_{y_1}"description]\ar[r,"q_1^+"]&x_2\ar[d,"n_{x_2}"description]\ar[r]&\cdots\ar[r]&x_n\ar[d,"n_{x_n}"description]\ar[r,"p_n^-"]&y_n\ar[d,"n_{y_n}"description]\ar[r,"q_n^-"]&y\ar[d,"n_{y}"description]\\
      \nf x\ar[r,equals]&\nf x\ar[r,equals]&\nf x\ar[r,equals]&\cdots\ar[r,equals]&\nf x\ar[r,equals]&\nf x\ar[r,equals]&\nf x
    \end{tikzcd}
  \]
  which allows us to conclude.
\end{proof}

\noindent
This implies the following ``abstract'' variant of Squier's homotopical theorem~\cite{guiraud2018polygraphs,lafont1995new,squier1994finiteness}:

\begin{proposition}
  \label{prop:ars-cr}
  \label{prop:ars-coh}
  If $\P$ is weakly $W$-normalizing and $W$-confluent then it is $W$-coherent.
\end{proposition}
\begin{proof}
  Given two parallel zig-zags $p,q:x\zzto y$ in $\freegpd W$, we have
  $p\cohto q$, since the following diagram commutes in $\pgpd\P$:
  \[
    \begin{tikzcd}
      &\ar[dl,"p"']x\ar[dd,"n_x"description]\ar[dr,"q"]&\\
      y\ar[ddr,bend right,"\id_y"']\ar[dr,"n_y"']&&\ar[ddl,bend left,"\id_y"]\ar[dl,"n_y"]y\\
      &\nf{y}\ar[d,"n_y^-"',pos=.2]&\\
      &y
    \end{tikzcd}
  \]
  Namely, we have $\nf x=\nf y$ by confluence, the two triangles above commute
  by \cref{prop:ars-zz-confl}, and the two triangles below do because $n_y^-$ is
  an inverse for $n_y$.
\end{proof}

\begin{example}
  As a variant of \cref{ex:rigid-equivalence}, consider the \ARS{2}~$\P$ with
  $\P_0=\set{x,y}$, $\P_1=\set{a,b:x\to y}$ and $\P_2=\emptyset$, \ie $
  \begin{tikzcd}
    x\ar[r,shift left,"a"]\ar[r,shift right,"b"']&y
  \end{tikzcd}
  $. With
  $W=\set{a}$, we have that $\P$ is $W$-terminating and locally $W$-confluent,
  thus $W$-confluent by~\cref{prop:lc+n-c}, and thus $W$-coherent by
  \cref{lem:term-wn,prop:ars-coh}. With $W=\set{a,b}$, we have seen in
  \cref{ex:rigid-equivalence} that the groupoid $\pgpd{W}$ is not rigid and,
  indeed, $\P$ is not $W$-confluent because $\pgpd a\neq\pgpd b$
  (because $\P_2=\emptyset$).
\end{example}

\begin{definition}
  We write $\NF{\pgpd\P}$ for the \emph{category of normal forms} of~$\pgpd\P$, defined as the full subcategory of~$\pgpd\P$ whose objects are those in $W$-normal form.
\end{definition}

\begin{lemma}
  \label{normal-subcat-equiv}
  If~$\P$ is weakly $W$-normalizing, then the inclusion functor $\NF{\pgpd\P}\to\pgpd\P$ is an equivalence of categories.
\end{lemma}
\begin{proof}
  An object~$x$ of~$\P$ admits a normal form $\nf{x}$, by
  \cref{lem:term-wn}. Writing $n_x:x\pathto\nf{x}$ for a normalization path, we
  have an isomorphism $\pgpd{n_x}:x\to\hat{x}$ in $\pgpd\P$. The inclusion
  functor is thus full and faithful (by definition), and every object
  of~$\pgpd\P$ is isomorphic to an object in the image, it is thus an
  equivalence of categories.
\end{proof}

\noindent
When $\P$ is $W$-convergent, the equivalence given in the above lemma is precisely the one with the quotient category:

\begin{proposition}
  \label{prop:quot-nf}
  If~$\P$ is $W$-convergent, then the quotient category~$\pgpd\P/W$ is isomorphic to the category of normal forms~$\NF{\pgpd\P}$.
\end{proposition}
\begin{proof}
  Since~$\P$ is $W$-convergent, by \cref{prop:ars-coh,coherent-rigid}, the groupoid generated by~$W$ is rigid and we thus have the description of the quotient $\pgpd\P/W$ given by \cref{prop:rigid-quotient}.
  We have a canonical functor $\NF{\pgpd\P}\to\pgpd\P/W$, obtained as the composite of the inclusion functor $\NF{\pgpd\P}\to\pgpd\P$ with the quotient functor $\pgpd\P\to\pgpd\P/W$. An object of~$\pgpd\P/W$ is an equivalence class $[x]$ of objects which, by convergence, contains a unique normal form, namely $\nf x$. The functor is bijective on objects.
  By weak normalization (\cref{lem:term-wn}), any morphism~$f:x\to y$ is equivalent to one with both normal source and target, namely $\nf f=n_y\circ f\circ n_x^-:\nf x\to\nf y$, hence the functor is full.
  Consider two morphisms $f,g:\nf x\to\nf y$ in $\NF\P$ with the same image $[f]=[g]$: by definition of the equivalence on morphisms, there exist morphisms $v:\nf x\to\nf x$ and $w:\nf y\to\nf y$ in~$\freegpd W$ making the diagram
  \[
    \begin{tikzcd}
      \nf x\ar[d,"v"']\ar[r,"f"]&\nf y\ar[d,"w"]\\
      \nf x\ar[r,"g"']&\nf y
    \end{tikzcd}
  \]
  commute. By the Church-Rosser property (\cref{prop:ars-cr}), we have $v=n_{\nf x}\circ n_{\nf x}^-$ and thus $v=\id_x$ (since $n_{\nf x}=\id_{\nf x}$ by hypothesis), and similarly $w=\id_y$. Hence $f=g$ and the functor is faithful.
  The functor is thus an isomorphism as being full, faithful and bijective on objects.
\end{proof}

We can finally summarize the results obtained in this section as follows. Given
a \ARS{2}~$\P$ and a set $W\subseteq\P_1$, we have the following possible
reasonable definitions of the fact that~$\P$ is \emph{coherent} \wrt
$W$:
\begin{enumerate}[(1)]
\item \label{ars-cr} Every parallel zig-zags with edges in~$W$ are equal\\
  (\ie the subgroupoid of $\pgpd P$ generated by~$W$ is rigid).
\item \label{ars-q} The quotient map $\pgpd\P\to \pgpd\P/W$ is an equivalence of
  categories.
\item \label{ars-nf} The inclusion~$\NF{\pgpd\P}\to\pgpd\P$ is an equivalence.
\item \label{ars-alg} The inclusion $\Alg(\pgpd\P/W,\D)\to\Alg(\pgpd\P,\D)$ is a
  natural equivalence of categories.
\end{enumerate}

\begin{theorem}
  \label{thm:ars-coh}
  If~$\P$ is $W$-convergent then all the above coherence properties hold.
\end{theorem}
\begin{proof}
  \ref{ars-cr} is given by \cref{prop:ars-coh}, \ref{ars-q} is given by
  \ref{ars-cr} and \cref{prop:rigid-equivalence}, \ref{ars-nf} is given by
  \cref{prop:quot-nf}, and \ref{ars-alg} is given by \ref{ars-cr} and
  \cref{prop:alg-str}.
\end{proof}


\subsection{Presenting the groupoid of normal forms}
We would now like to provide an explicit description of $\NF{\pgpd\P}$, by a \ARS2. A good candidate is the following \ARS2 $\P\setminus W$ obtained by ``restricting $\P$ to normal forms''. More precisely, it consists of
\begin{itemize}
\item $(\P\setminus W)_0$: the objects of $\P\setminus W$ are those of~$\P$ in $W$-normal form,
\item $(\P\setminus W)_1$: the rewriting rules of~$\P\setminus W$ are those of~$\P$ whose source and target are both in~$(\P\setminus W)_0$ (in particular, it does not contain any element of~$W$, thus the notation),
\item $(P\setminus W)_2$: the coherence relations are those of $\P_2$ whose source and target both belong to~$\freegpd{(\P\setminus W)_1}$.
\end{itemize}
It is not the case in general that this \ARS2 presents the category $\NF{\pgpd\P}$, but we will provide here conditions which ensure that this holds, see also~\cite{clerc2015presenting,mimram2017coherent} for alternative conditions. For simplicity, we suppose here that the source and target of every rewriting step in $\P_2$ is a path (as opposed to a zig-zag).

\begin{proposition}
  \label{prop:ars-nf-pres}
  \label{ars-nf-pres}
  Suppose that
  \begin{enumerate}
  \item \label{ars-nf-pres-cv} $\P$ is $W$-convergent,
  \item \label{ars-nf-pres-nt} every rule $a:x\to y$ in~$\P_1$ whose source~$x$
    is $W$-normal also has a $W$-normal target~$y$,
  \item \label{ars-nf-pres-res} for every coinitial rule~$a:x\to y$ in $\P_1$
    and path $w:x\pathto x'$ in $\freecat W$, there are paths $p:x'\pathto y'$
    in $\freecat{\P_1}$ and $w':y\pathto y'\in\freecat W$ such that
    $a\pcomp w'\cohto w\pcomp p$:\TODO{il suffit ici de le supposer pour~$x'$ en forme normale}
    \[
      \begin{tikzcd}
        x\ar[d,"w"',"*"]\ar[r,"a"]&y\ar[d,dotted,"*"',"w'"]\\
        x'\ar[r,dotted,"*","p"']\ar[ur,Leftrightarrow,dotted,"*",shorten=5pt]&y'
      \end{tikzcd}
    \]
  \item \label{ars-nf-pres-res-coh} for every coherence relation
    $A:p\To q:x\pathto y$, and for every path $w:x\pathto x'$, the paths
    $p',q':x'\pathto y'$ in $\freecat\P_1$ and $w':y\pathto y'$ in $\freecat W$
    such that $p\pcomp w'\cohto w\pcomp p'$ and $q\pcomp w'\cohto w\pcomp q'$
    induced by \cref{ars-nf-pres-res} satisfy $p'\cohto q'$:
    \[
      \begin{tikzcd}
        x\ar[d,"w"',"*"]\ar[r,bend left,"p"{name=p}]\ar[r,bend right,"q"'{name=q}]&y\ar[d,dotted,"*"',"w'"]\\
        x'\ar[r,bend left,dotted,"p'"{name=p'}]\ar[r,bend right,dotted,"q'"'{name=q'}]&y'
        \ar[from=p,to=q,Rightarrow,shorten=4pt,"A"']
        \ar[from=p',to=q',Leftrightarrow,dotted,shorten=2pt,"*"']
      \end{tikzcd}
    \]
  \end{enumerate}
  Then $\NF{\pgpd\P}$ is isomorphic to $\pgpd{\P\setminus W}$.
\end{proposition}
\begin{proof}  
  We write $\Q=\P\setminus W$. Since $\Q$ is, by definition, a sub-\ARS2 of~$\P$
  there is a canonical functor $\pgpd\Q\to\pgpd\P$. Moreover, since the objects
  of $\P$ are, by definition, in $W$-normal form, this functor corestricts to a
  functor $F:\pgpd\Q\to\NF{\pgpd\P}$ which is the identity on objects.

  First, note that condition (\ref{ars-nf-pres-nt}) implies that for any
  path~$p$ of the form
  \[
    \begin{tikzcd}
      x_0\ar[r,"a_1"]&x_1\ar[r,"a_2"]&\cdots\ar[r,"a_n"]&x_n
    \end{tikzcd}
  \]
  with $x_0$ in $W$-normal form, we have that all the $x_i$ are in $W$-normal
  form and thus $p$ belongs to $\freecat{\Q_1}$. Similarly, every coherence
  relation $A:p\To q:x\pathto y$ with $x$ in $W$-normal form belongs to~$\Q_2$.

  We claim that for every zig-zag $p:x\zzto y$ in $\freegpd\P_1$ there is
  zig-zag~$q\in\freegpd\Q_1$ such that $p\cohto n_x\pcomp q\pcomp n_y^-$. We
  have that $p$ is of the form
  $p=w_0\pcomp a_1\pcomp w_1\pcomp a_2\pcomp w_2\pcomp\ldots\pcomp a_n\pcomp
  w_n$ where the $a_i$ are rules in~$\P_1$ which are not in $W$ (possibly taken
  backward) and the $w_i$ are in $\freegpd W$. For instance, consider the case
  $n=1$ and a path~$p$ of the form $p=v\pcomp a\pcomp w$ with $a\in\Q_1$ and
  $v,w\in\freegpd W$ (the case where $a$ is reversed is similar, and the general
  case follows by induction):\TODO{on pourrait le faire step by step sans différencier $a$ dans $W$ ou pas...}
  \[
    \begin{tikzcd}
      x\ar[dr,"n_x"']\ar[rr,"v"]&&\ar[dl,"n_{x'}"]x'\ar[r,"a"]&y'\ar[d,"w'"']\ar[dr,"n_{y'}"']\ar[rr,"w"]&&\ar[dl,"n_y"]y\\
      &\nf x\ar[rr,dotted,"q"']&&y''\ar[r,equals]&\nf y
    \end{tikzcd}
  \]
  By hypothesis~\cref{ars-nf-pres-cv} and \cref{prop:ars-zz-confl}, we have
  $v\cohto n_x\pcomp n_{x'}^-$ and $w\cohto n_{y'}\pcomp n_y^-$.
  By hypothesis~\cref{ars-nf-pres-res}, there exist paths $q:\nf x\pathto y''$
  in $\freecat\P_1$ and $w':y'\pathto y''$ in $\freecat W$ such that
  $a\pcomp w'\cohto n_{x'}\pcomp q$.
  By hypothesis~\cref{ars-nf-pres-nt} and the remark at the beginning of this
  proof, we have that $q\in\freecat{\Q_1}$ and $y''$ is a normal form. By
  \cref{ars-nf-pres-cv}, we thus have $y''=\nf y$ and $w'\cohto n_{y'}$, and we
  conclude.
  As a particular case of the property we have just shown, for any zig-zag
  $p:x\zzto y$ whose source and target are in $W$-normal form, we have that that
  $p$ is equivalent to a zig-zag $q:x\zzto y$ (since in this case both $n_x$ and
  $n_y$ are identities). The functor $F:\pgpd\Q\to\NF{\pgpd\P}$ is thus full.

  In the following, given a path $p:x\to y$ in $\freecat\P_1$, we write
  $\nf p:\nf x\to\nf y$ in $\freecat\Q_1$ for a path such that
  $p=n_x\pcomp\nf p\pcomp n_y^-$. Such a path always exists by the previous
  reasoning and can be constructed in a functorial way (\ie
  $\widehat{p_1\pcomp p_2}=\nf p_1\pcomp\nf p_2$).
  Now, suppose given two zig-zags $p,p':x\zzto y$ in $\freegpd\P_1$ such that
  $p\cohto p'$. The relation $p\cohto p'$ means that there is a sequence
  $p_1,\ldots,p_n$ of zig-zags in $\freegpd\P_1$ such that $p_1=p$, $p_n=p'$ and
  each $p_i$ is related to $p_{i+1}$ by taking a relation in context, as in
  \cref{coh-step}:
  \[
    p=p_1\Leftrightarrow p_2\Leftrightarrow\ldots\Leftrightarrow p_n=p'
  \]
  More formally, for $1\leq i<n$, there is a decomposition
  \begin{align*}
    p_i&=q_i\pcomp r_i\pcomp s_i
    &
    p_{i+1}&=q_i\pcomp r_i'\pcomp s_i
  \end{align*}
  such that there is a relation $A:r_i\To r_i'$ or $A:r_i'\To r_i$ (another
  approach would consist in reasoning by induction on the $2$-cells of the
  freely generated $2$-groupoid of \cref{free-2-groupoid}). By
  hypothesis~\cref{ars-nf-pres-res-coh}, there is a relation
  $\nf r_i\cohto\nf r_i'$, and thus $\nf p_i\cohto\nf p_{i+1}$ by
  functoriality. By recurrence on~$n$, we thus have $\nf p\cohto\nf p'$.
  From this, we deduce that that the functor $F$ is also faithful.
\end{proof}

\noindent
In practice, condition \cref{ars-nf-pres-cv} can be shown using traditional
rewriting techniques (\eg \cref{prop:lc+n-c}) and condition
\cref{ars-nf-pres-nt} is easily checked by direct inspection of the rewriting
rules. We provide below sufficient conditions in order to show the two remaining
conditions:

\begin{proposition}
  \label{ars-nf-pres-local}
  We have the following.
  \begin{enumerate}
  \item[\cref{ars-nf-pres-res}] Suppose that for every coinitial
    rules~$a:x\to y$ in $\P_1$ and $w:x\to x'$ in $W$, there are paths
    $p:x'\pathto y'$ in $\freecat\P_1$ and $w':y\pathto y'$ in $\freecat W$ such
    that $w$ is of length at most one and $a\pcomp w'\cohto w\pcomp p$:
    \begin{equation}
      \label{eq:ars-nf-pres-res}
      \begin{tikzcd}
        x\ar[d,"w"']\ar[r,"a"]&y\ar[d,dotted,"w'"]\\
        x'\ar[r,dotted,"*","p"']&y'
      \end{tikzcd}
      \qquad\qquad
      \text{or}
      \qquad\qquad
      \begin{tikzcd}
        x\ar[d,"w"']\ar[r,"a"]&y\ar[d,dotted,equals]\\
        x'\ar[r,dotted,"*","p"']&y'
      \end{tikzcd}
    \end{equation}
    The condition \cref{ars-nf-pres-res} of \cref{prop:ars-nf-pres} is
    satisfied.
  \item[\cref{ars-nf-pres-res-coh}] Suppose that condition
    \cref{ars-nf-pres-res} is satisfied and for every coherence relation
    $A:p\To q:x\pathto y$ in $\P_2$ and rule $w:x\to x'$ in~$W$, the paths
    $p',q':x'\pathto y'$ in $\freecat\P_1$ and $w':y\pathto y'$ in~$\freecat W$
    of length at most one such that $p\pcomp w'\cohto w\pcomp p'$ and
    $q\pcomp w'\cohto w\pcomp q'$ induced by \cref{ars-nf-pres-res} are such
    that $p'\cohto q'$:
    \begin{equation}
      \label{eq:ars-nf-pres-res-coh}
      \begin{tikzcd}
        x\ar[d,"w"',"*"]\ar[r,bend left,"p"{name=p}]\ar[r,bend right,"q"'{name=q}]&y\ar[d,dotted,"w'"]\\
        x'\ar[r,bend left,dotted,"p'"{name=p'}]\ar[r,bend right,dotted,"q'"'{name=q'}]&y'
        \ar[from=p,to=q,Rightarrow,shorten=4pt,"A"']
        \ar[from=p',to=q',Leftrightarrow,dotted,shorten=2pt,"*"']
      \end{tikzcd}
      \qquad\qquad
      \text{or}
      \qquad\qquad
      \begin{tikzcd}
        x\ar[d,"w"']\ar[r,bend left,"p"{name=p}]\ar[r,bend right,"q"'{name=q}]&y\ar[d,equals]\\
        x'\ar[r,bend left,"p'"{name=p'}]\ar[r,bend right,"q'"'{name=q'}]&y
        \ar[from=p,to=q,Rightarrow,shorten=4pt,"A"']
        \ar[from=p',to=q',Leftrightarrow,dotted,shorten=2pt,"*"']
      \end{tikzcd}
    \end{equation}
    Then condition \cref{ars-nf-pres-res-coh} of \cref{prop:ars-nf-pres} is
    satisfied.
  \end{enumerate}
\end{proposition}
\begin{proof}
  Both properties are easily shown by recurrence on the length of~$w$.
\end{proof}

\section{Relative coherence for Lawvere theories}
\label{sec:coh-lt}
In order to use the developments of \cref{sec:ars} in concrete situations, such as (symmetric) monoidal categories, we need to consider a more structured notion of theory. For this reason, we consider here Lawvere 2-theories, as well as the adapted notion of rewriting, which is a coherent extension of the traditional notion of term rewriting systems.

\subsection{Lawvere 2-theories}
We begin by recalling the traditional notion due to
Lawvere~\cite{lawvere1963functorial}:

\begin{definition}
  A \emph{Lawvere theory}~$\T$ is a cartesian category, with $\N$ as set of
  objects, and cartesian product given on objects by addition.
  A morphism between Lawvere theories is a product-preserving functor and we
  write $\nLaw1$ for the category of Lawvere theories.
\end{definition}

\noindent
For simplicity, we restrict here to unsorted theories, but the developments
performed here could easily adapted to the multi-sorted case. In such a theory,
we usually restrict our attention to morphisms with $1$ as codomain, since
$\T(n,m)\isoto\T(n,1)^m$ by cartesianness.

A \emph{$(2,1)$-category} is a $2$-category in which every $2$-cell is invertible, \ie a category enriched in groupoids. The following generalization of Lawvere theories was introduced in various places, see~\cite{gray19732,yanofsky2000syntax,yanofsky2001coherence}, as well as~\cite{power1999enriched} for the enriched point of view:

\begin{definition}
  A \emph{Lawvere 2-theory}~$\T$ is a cartesian $(2,1)$-category with~$\N$ as objects, and cartesian product given on objects by addition.
  A morphism $F:\T\to\U$ between $2$-theories is a $2$-functor which preserves products. We write $\nLaw2$ for the resulting category (which can be extended to a $3$-category by respectively taking natural transformations and modifications as $2$- and $3$-cells).
\end{definition}

We can reuse the properties developed in \cref{sec:coh-cat,sec:ars} by working ``hom-wise'' as follows.
Let~$\T$ be a $2$-theory together with a subset~$W$ of the $2$-cells. We write $\W$ for the sub-$2$-theory of~$\T$, with the same $0$- and $1$-cells, and whose $2$-cells contain~$W$ (we often assimilate this $2$-theory to its set of $2$-cells).
A morphism $F:\T\to\U$ of Lawvere $2$-theories is \emph{$W$-strict} when it sends every $2$-cell in $W$ to an identity.

\begin{definition}
  The \emph{quotient $2$-theory} $\T/W$ is the theory equipped with a $W$-strict
  morphism $\T\to\T/W$ such that every $W$-strict morphism $F:\T\to\U$ extends
  uniquely as a morphism $\T/W\to\U$:
  \[
    \begin{tikzcd}
      \T\ar[d]\ar[r,"F"]&\U\\
      \T/W\ar[ur,dotted,"\tilde F"']
    \end{tikzcd}
  \]
\end{definition}

\noindent
Such a quotient $2$-theory always exists by general arguments: $2$-theories form a locally presentable category, which therefore has colimits, and thus coequalizers~\cite{rosicky1994locally}. We have $\T/W\isoto\T/\W$, so that we can always assume that we are quotienting by a sub-$2$-theory. On hom-categories, the quotient corresponds to the one introduced in \cref{sec:ars-quot}:

\begin{lemma}
  For every $m,n\in\N$, we have
  \[
    (\T/\W)(m,n)=\T(m,n)/\W(m,n)
    \text.
  \]
\end{lemma}

We say that a morphism
\[
  F:\T\to\U
\]
is a \emph{local equivalence} when for every objects $m,n\in\T$, the induced functor
\[
  F_{m,n}:\T(m,n)\to\U(m,n)
\]
between hom-categories is an equivalence.

\begin{definition}
  A theory $\W$ is \emph{$2$-rigid} when any two parallel $2$-cells are equal.
\end{definition}

\noindent
Note that a theory $\W$ is $2$-rigid if and only if the category $\W(m,n)$ is rigid for every $0$-cells~$m$ and~$n$.
By direct application of \cref{prop:rigid-equivalence}, we have

\begin{theorem}
  \label{prop:rigid-2equivalence}
  The quotient $2$-functor $\T\to\T/\W$ is a local equivalence iff $\W$ is $2$-rigid.
\end{theorem}



\subsection{Algebras for Lawvere 2-theories}
The notion of algebra for $2$-theories was extensively studied by
Yanofsky~\cite{yanofsky2000syntax,yanofsky2001coherence}, we refer to his work
for details.

\begin{definition}
  An \emph{algebra} for a Lawvere $2$-theory~$\T$ is a $2$-functor $C:\T\to\Cat$ which preserves products. By abuse of notation, we often write $C$ instead of $C\,1$ and suppose that products are strictly preserved, so that $Cn=C^n$.
\end{definition}

\begin{definition}
  A \emph{pseudo-natural transformation} $F:C\To D$ between algebras~$C$ and~$D$ consists in a functor $F:C\to D$ together with a family $\phi_f:Df\circ F^n\To F\circ Cf$ of natural transformations indexed by
  $1$-cells $f:n\to 1$ in $\T$,
  \[
    \begin{tikzcd}
      C^n\ar[d,"F^n"']\ar[r,"Cf"]&C\ar[d,"F"d]\\
      D^n\ar[r,"Df"']\ar[ur,Rightarrow,shorten=2ex,"\phi_f"]&D
    \end{tikzcd}
  \]
  which is compatible with products, composition and $2$-cells of~$\T$.
\end{definition}

\begin{definition}
  A \emph{modification} $\mu:F\TO G:C\To D$ between two pseudo-natural transformations is a natural transformation $\mu:F\TO G$ which is compatible with $2$-cells of~$\T$. We write $\Alg(\T)$ for the 2-category of algebras of a $2$-theory~$\T$, pseudo-natural transformations and modifications.
\end{definition}

\begin{example}
  \label{ex:trs-alg}
  Consider the \TRS{2} $\Mon$ of \cref{ex:rs-mon} below. The $2$-category $\Alg(\pgpd\Mon)$ of algebras of the presented $2$-theory is isomorphic to the 2-category~$\MonCat$ of monoidal categories, strong monoidal functors and monoidal natural transformations. It might be surprising that~$\Mon$ has five coherence relations whereas the traditional definition of monoidal categories only features two axioms, which correspond to the coherence relations~$A$ and~$C$. There is no contradiction here: the commutation of the two axioms can be shown to imply the one of the three other~\cite{guiraud2012coherence,kelly1964maclane}.

  Similarly, writing $\W=\pgpd\Mon$, the 2-category $\Alg(\pgpd\Mon/\W)$ is the sub-2-category of $\Alg(\pgpd\Mon)$ whose algebras are monoidal categories where the coherence natural transformations are identities, \ie the 2-category $\StrMonCat$ of strict monoidal categories.
\end{example}

We conjecture that one can generalize the classical proof that any monoidal category is monoidally equivalent to a strict one~\cite[Theorem XI.3.1]{maclane1998categories} to show the following general~(C3) coherence theorem, as well as its (C4) generalization:

\begin{conjecture}[C3]
  \label{conj:strict-equiv}
  When $\W$ is 2-rigid, every~$\T$-algebra is equivalent to a $\T/\W$ algebra.
\end{conjecture}

\begin{conjecture}[C4]
  \label{conj:strict-adj}
  When $\W$ is 2-rigid, the 2-functor $\Alg(\T/\W)\to\Alg(\T)$ induced by precomposition with the quotient 2-functor $\T\to\T/\W$ has a left adjoint such that the components of the unit are equivalences.
\end{conjecture}

\noindent
By \cref{ex:trs-alg}, the conjectures would allow recovering the coherence theorem for monoidal categories by taking $\W=\T=\pgpd\Mon$.
We will see in \cref{smc-coh} that the conjectures would also allow recovering the coherence theorem for symmetric monoidal categories by taking $\T=\pgpd\SMon$ to be the 2-theory presented by a suitable \TRS2, and $\W$ the subtheory generated by the generating $2$-cells corresponding to $\alpha$, $\lambda$ and $\rho$. One should also be able to obtain the coherence theorems for braided monoidal categories in a similar way.

A detailed study of those conjectures is left for future works, since it would require introducing some more categorical material, and our aim in this article is to focus on the rewriting techniques. Note that, apart from informal explanations, we could not find a proof of \cref{conj:strict-equiv,conj:strict-adj} for symmetric or braided monoidal categories in the literature, \eg in~\cite{joyal1993braided,maclane1998categories,maclane1963natural} (in~\cite[Theorem 2.5]{joyal1993braided} the result is only shown for free braided monoidal categories). Once proved, it would be interesting to investigate whether the converse implications hold, \ie whether the conditions imply $\W$ being 2-rigid.

\begin{remark}
  One could hope for the following alternative generalization of \cref{prop:equiv-alg}: \emph{a 2-functor~$F:\T\to\T'$ between theories is a biequivalence if and only if the functor $\Alg(F):\Alg(\T')\to\Alg(\T)$ induced by precomposition is an equivalence}. This is claimed in~\cite[Proposition~7]{yanofsky2001coherence}, along with the corollary that the categories $\Alg(\T/\W)$ and $\Alg(\T)$ are equivalent when $\W$ is 2-rigid (at least in the particular case where $\T$ is the theory of monoids and $\W=\T$, as in \cref{ex:trs-alg}). However, both proofs are based on the wrong claim that biessentially surjective local equivalences coincide with biequivalences~\cite[Proposition~6]{yanofsky2001coherence}.

  We recall that a 2-functor $F:\C\to\D$ between 2-categories is
  \begin{itemize}
  \item \emph{biessentially surjective} when for every $0$-cell $x$ of~$\C$ there is a $0$-cell~$y$ of $\D$ together with an equivalence $Fx\equivto y$,
  \item a \emph{local equivalence} when for every $0$-cells $x,y$ of~$\C$, the induced functor $\C(x,y)\to\C(Fx,Fy)$ is an equivalence,
  \item a \emph{biequivalence} when there is a $2$-functor $G:\D\to\C$ together with natural equivalences $G\circ F\equivto\Id_\C$ and $F\circ G\equivto\Id_\D$.
  \end{itemize}
  Any biequivalence is a biessentially surjective local equivalence, but the converse is not true: from a biessentially surjective local equivalence $F:\C\to\D$, one can in general only construct a pseudofunctor (as opposed to a 2-functor) $G:\D\to\C$ satisfying the desired properties. A concrete counter-example is given in~\cite[Example 3.1]{lack2002quillen}, as we now recall. Consider the 2-categories
  \begin{itemize}
  \item $\C$ with one 0-cell $\pt$, with $(\N,+,0)$ as monoid of endomorphisms on~$\pt$, and one 2-cell between 1-cells $m,n\in\N$ whenever $m$ and $n$ have the same parity,
  \item $\D$ with one 0-cell $\pt$, with $(\N/2\N,+,0)$ as monoid of endomorphisms on~$\pt$, and only trivial 2-cells.
  \end{itemize}
  In other terms, the 2-category~$\C$ is freely generated by one 0-cell $\pt$, one 1-cell $1$ and 2-cells are the congruence generated by $1+1=0$:
  
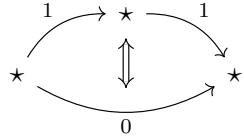
\[
    \begin{tikzcd}[row sep=small]
      &\pt\ar[phantom,""{below,name=A}]\ar[dr,bend left,"1"]\\
      \pt\ar[ur,bend left,"1"]\ar[rr,bend right,"0"',""{name=B}]&{}&\pt
      \ar[from=A,to=B,Leftrightarrow,shorten=4pt]
    \end{tikzcd}
  \]
  and the category $\D$ is obtained from~$\C$ by formally turning 2-cells into identities. We thus have a quotient 2-functor $F:\C\to\D$ sending a 1-cell $n\in\N$ to $0$ or~$1$ depending on whether~$n$ is even or odd, \ie $F({2n})=0$ and $F({2n+1})=1$. Conversely, the functor $G:\D\to\C$ should associate to every 1-cell of~$\D$ a representative, \ie $G(0)={2m}$ and $G(1)={2n+1}$ for some $m,n\in\N$. Since we require that $G$ is functorial, we should have
  \[
    2m=G(0)=G(0+0)=G(0)+G(0)=4m
  \]
  so that $m=0$, and
  \[
    0=G(0)=G(1+1)=G(1)+G(1)=2n+2
  \]
  thus reaching a contradiction. The morale is that a choice of representatives is usually not strictly functorial, but only pseudo-functorial (above, $0$ and $2n+2$ are not equal, but they are in the same equivalence class).

  Whether the above generalization of \cref{prop:equiv-alg} holds is left open.
  However, we cannot use it to conclude that, when $\W\subseteq\T$ is 2-rigid, we have that the canonical functor $\Alg(\T/\W)\to\Alg(\T)$ is an equivalence, because the functor $\T\to\T/\W$ is in general a biessentially surjective local equivalence (\cref{prop:rigid-2equivalence}), but not a biequivalence.
  %
\end{remark}







\section{Coherent term rewriting systems}
\label{sec:trs}
\subsection{Extended rewriting systems}
We now recall the categorical setting for term rewriting systems, as well as their extension in order to handle coherence. A more detailed presentation can be found in~\cite{polygraphs,beke2011categorification,cohen2009coherence,malbos2016homological}.

\begin{definition}
  A \emph{signature} consists in a set~$\S_1$ of \emph{symbols} together with a function $s_0:\S_1\to\N$ associating to each symbol an \emph{arity}, and we write $a:n\to 1$ for a symbol~$a$ of arity~$n$. A morphism of signatures is a function between the corresponding sets of symbols which preserves arity, and we write $\nCPol1$ for the corresponding category.
\end{definition}

\begin{remark}
  If we were interested in the multi-sorted case, our signature would rather consist in a set~$\S_1$ together with a set $\S_0$ of \emph{sorts}, along with functions $s_0:\S_1\to\freecat\S_0$ (where $\freecat\S_0$ is the free monoid on $\S_0$) and $t_0:\S_1\to\S_0$, respectively indicating the sorts of the inputs and of the output. The above definition can be recovered as the particular case where $\S_0$ is the terminal set. This point of view explains why the index fo~$\S_1$ is~$1$.
\end{remark}

\noindent
There is a forgetful functor $\nLaw1\to\nCPol1$, sending a theory~$\T$ to the
set $\bigsqcup_{n\in\N}\T(n,1)$ with first projection as arity. This functor
admits a left adjoint $-^*:\nCPol1\to\nLaw1$, which we now describe.
Given a signature $\S_1$, and $n\in\N$, $\S_1^*(n,1)$ is the set of \emph{terms}
of arity~$n$: those are formed using operations, with variables
in~$\set{x_1^n,x_2^n,\ldots,x_n^n}$. Note that the superscript for variables is
necessary to unambiguously recover the type of a variable, \ie $x^n_i:n\to 1$,
but for simplicity we will often omit it in the following. More explicitly, the
family of sets $\S_1^*(n,1)$ is the smallest one such that
\begin{itemize}
\item for $1\leq i\leq n$, we have
  \[
    x^n_i\in\S_1^*(n,1)
  \]
\item given $m,n\in\N$, a symbol $a:n\to 1$ and terms
  $t_1,\ldots,t_n\in\S_1^*(m,1)$, we have
  \[
    a(t_1,\ldots,t_m)\in\S_1^*(m,1)
  \]
\end{itemize}
More generally, a morphism~$f$ in $\S_1^*(n,m)$ is an $m$-uple
\[
  f
  =
  \tuple{t_1,\ldots,t_m}
\]
of terms~$t_i$ with variables in $\set{x_1^n,\ldots,x_n^n}$, which can be
thought of as a formal \emph{substitution}. Given such a substitution~$f$ and a
term $t$, we write
\[
  t[f]
  \qquad\qquad\text{or}\qquad\qquad
  t[t_1/x_1,\ldots,t_n/x_n]
\]
for the term obtained from $t$ by formally replacing each variable $x_i^n$
by~$t_i$. This operation is thus defined inductively by
\begin{align*}
  x_i^n[f]&=t_i
  &
  a(u_1,\ldots,u_k)[f]&=a(u_1[f],\ldots,u_k[f])
\end{align*}
The composition of two morphisms $\tuple{t_1,\ldots,t_m}:\S_1^*(n,m)$ and
$\tuple{u_1,\ldots,u_k}:\S_1(m,k)$ is given by parallel substitution:
\[
  \tuple{u_1,\ldots,u_k}\circ\tuple{t_1,\ldots,t_m}=\tuple{u_1[t_1/x_1,\ldots,t_n/x_n],\ldots,u_k[t_1/x_1,\ldots,t_m/x_m]}
\]
and the identity in $\freecat\S_1(n,n)$ is $\tuple{x_1^n,\ldots,x_n^n}$.
The resulting category $\freecat\S_1$ is easily checked to be a Lawvere theory,
which satisfies the following universal property:

\begin{lemma}
  The Lawvere theory $\freecat\S_1$ is the free Lawvere theory on the
  signature~$\S_1$.
\end{lemma}

\noindent
By abuse of notation, we sometimes write
\[
  \freecat\S_1
  =
  \bigsqcup_{m,n\in\N}\S_1^*(m,n)
\]
for the set of all substitutions and
$\freecat s_0,\freecat t_0:\freecat\S_1\to\N$ for the source and target maps,
and $i_1:\S_1\to\freecat\S_1$ for the map sending an operation $a:n\to 1$ to the
substitution consisting of one term $\tuple{a(x_1^n,\ldots,x_n^n)}$, so that we
have $\freecat s_0\circ i_1=s_0$ and $\freecat t_0\circ i_1=1$.

\begin{definition}
  \label{trs}
  A \emph{term rewriting system}, or \TRS{}, $\S$ consists of a signature~$\S_1$ together with a set~$\S_2$ of \emph{rewriting rules} and functions $s_1,t_1:\S_2\to\freecat\S_1$ which indicate the source and target of each rewriting rule, and are supposed to satisfy
  \begin{align*}
    \freecat s_0\circ s_1&=\freecat s_0\circ t_1
    &
    \freecat t_0\circ s_1&=\freecat t_0\circ t_1=1    
  \end{align*}
  This data can be summarized in the following diagram:
  \[
    \begin{tikzcd}
      &\ar[dl,"a"']\ar[d,"i_1"']\S_1&\ar[dl,shift right,"s_1"']\ar[dl,shift left,"t_1"]\S_2\\
      \N&\ar[l,shift right,"\freecat s_0"']\ar[l,shift left,"\freecat t_0"]\freecat\S_1
    \end{tikzcd}
  \]
\end{definition}


\noindent
We sometimes write
\[
  \rho:t\To u
\]
for a rule $\rho$ with $t$ as source and $u$ as target. The relations satisfied by any \TRS{} ensure that both~$t$ and~$u$ have the same arity.

We now need to introduce some notions in order to be able to define rewriting in this setting. In case it helps, those are illustrated in \cref{rs-mon} below.
A \emph{context}~$C$ of arity~$n$ is a term with variables in
$\set{x_1,\ldots,x_n,\square}$ where the variable~$\hole$ is a particular
variable, the \emph{hole}, occurring exactly once. Here, we define the number
$\occurrences it$ of occurrences of a variable $x_i$ (and similarly for $\hole$)
in a term $t$ by induction by
\begin{align*}
  \occurrences i{x_i}&=1
  &
  \occurrences i{a(t_1,\ldots,t_n)}&=\sum_{k=1}^n\occurrences i{t_k}
\end{align*}
We write $\S^\hole_n$ for the set of contexts of arity~$n$. Given a context~$C$
and a term~$t$, both of same arity~$n$, we write $C[t]$ for the term obtained
from~$C$ by replacing $\hole$ by~$t$. The composition of contexts~$C$ and $D$ is
given by substitution
\[
  D\circ C=D[C]
\]
This composition is associative and admits the identity context $\hole$ as
neutral element.
A \emph{bicontext} from $n$ to $k$, is a pair $(C,f)$ consisting of a context~$C$ of arity $n$ and a substitution~$f\in\freecat\S_1(n,k)$. This data can be thought of as the specification of a function on terms
\begin{align*}
  \freecat\S_1(n,1)&\to\freecat\S_1(k,1)\\
  \tuple t&\mapsto C[\tuple t\circ f]
\end{align*}
In the following, for simplicity, we will omit the brackets and simply write~$t$ instead of~$\tuple t$ in such an expression, so that the image of the function can also be denoted $C[t\circ f]$. This function will be referred as the \emph{action} of a bicontext on terms.
The composition of bicontexts~$(C,f)$ and~$(D,g)$ of suitable types is given by $(D\circ C,f\circ g)$. The action is compatible with this composition, in the sense that we have
\[
  D[C[-\circ f]\circ g]
  =
  (D\circ C[f])[-\circ(f\circ g)]
\]

A \emph{rewriting step} of arity~$n$
\[
  C[\rho\circ f]
  :
  C[t\circ f]
  \To
  C[u\circ f]
\]
is a triple consisting of
\begin{itemize}
\item a rewriting rule $\rho:t\To u$, with $t$ and $u$ of arity~$k$,
\item a context $C$ of arity $n$,
\item a substitution $f:n\to k$ in $\freecat\S_1$.
\end{itemize}
A rewriting step can thus be thought of as a rewriting rule in a bicontext. Its source is the term $C[t\circ f]$ and its target is the term $C[u\circ f]$. We write $\steps\S_2$ for the set of rewriting steps.
A \emph{rewriting path}~$\pi$ is a composable sequence
\[
  \begin{tikzcd}[sep=large]
    C_1[t_1\circ f_1]
    \ar[r,Rightarrow,"{C_1[\rho_1\circ f_1]}"]
    &
    C_1[u_1\circ f_1]
    =
    C_2[t_2\circ f_2]
    \ar[r,Rightarrow,"{C_2[\rho_2\circ f_2]}"]
    &
    \cdots
    \ar[r,Rightarrow,"{C_n[\rho_n\circ f_n]}"]
    &
    C_n[u_n\circ f_n]
  \end{tikzcd}
\]
of rewriting steps. We write~$\freecat\S_2$ for the set of rewriting paths and
adopt the previous notation, \eg we write $\pi\pcomp\pi'$ for the concatenation
of two composable rewriting paths $\pi$ and $\pi'$.
As in \cref{sec:coh-ars}, we can also define a notion of \emph{rewriting zig-zag} which is similar to rewriting paths excepting that some rewriting steps may be taken backwards, and write $\freegpd\S_2$ for the corresponding set. We sometimes write $\pi:t\Pathto u$ (\resp $\pi:t\Zzto u$) to indicate the source and target or a rewriting path~$\pi$.

Given a signature $\S_1$, there is a forgetful functor from the category of Lawvere $2$-theories with $\freecat\S_1$ as underlying Lawvere theory to the category rewriting systems with~$\S_1$ as signature (with the expected notion of morphism).

\begin{lemma}
  Given a \TRS{} $\S$, the Lawvere $2$-theory $\freegpd\S$ with $\freecat\S_1$ as $1$-cells and $\freegpd\S_2$ as $2$-cells is free on~$\S$.
\end{lemma}


The action of bicontexts on terms extend to rewriting steps as follows. Given a
rewriting step
\[
  C[\rho\circ f]
  :
  C[t\circ f]
  \To
  C[u\circ f]
\]
a context~$D$ and a substitution~$g$ of suitable types, we define
$D[C[\rho\circ f]\circ g]$ to be the rewriting step
\[
  (D\circ C[g])[\rho\circ(f\circ g)]
  :
  (D\circ C[g])[t\circ(f\circ g)]
  \To
  (D\circ C[g])[u\circ(f\circ g)]
\]
Moreover, we extend this action to rewriting paths and zig-zags by
functoriality, \ie
\[
  C[(p\pcomp q)\circ f]=C[p\circ f]\pcomp C[q\circ f]
\]

\begin{definition}
  \label{2trs}
  An \emph{extended term rewriting system}, or \TRS{2}, consists of a term
  rewriting system as above, together with a set $\S_3$ of \emph{coherence
    relations} and functions $s_2,t_2:\S_3\to\freegpd\S_2$, indicating their
  source and target, satisfying
  \begin{align*}
    \freegpd s_1\circ s_2&=\freegpd s_1\circ t_2
    &
    \freegpd t_1\circ s_2&=\freegpd t_1\circ t_2    
  \end{align*}
  Diagrammatically,
  \[
    \begin{tikzcd}
      &\ar[dl,"a"']\ar[d,"i_1"']\S_1&\ar[dl,shift right,"s_1"']\ar[dl,shift left,"t_1"]\S_2\ar[d,"i_2"']&\ar[dl,shift right,"s_2"']\ar[dl,shift left,"t_2"]\S_3\\
      \N&\ar[l,shift right,"\freecat s_0"']\ar[l,shift left,"\freecat t_0"]\freecat\S_1&\ar[l,shift right,"\freegpd s_1"']\ar[l,shift left,"\freegpd t_1"]\freegpd\S_2
    \end{tikzcd}
  \]
\end{definition}

\noindent
Given a \TRS2 as above, we sometimes write
\[
  A:\pi\TO\pi'
\]
to indicate that $A$ is a coherence relation in~$\S_3$ with $\pi$ as source and
$\pi'$ as target.
Given two rewriting paths~$\pi$ and~$\pi'$, we write $\pi\Cohto\pi'$ when they
are related by the smallest congruence identifying the source and target of any
coherence relation.

\begin{definition}
  The Lawvere $2$-theory \emph{presented} by a \TRS{2}~$\S$ is the $(2,1)$-category noted~$\pgpd\S$, with $\N$ as $0$-cells, $\freecat\S_1$ as $1$-cells and the quotient of $\freegpd\S_2$ under the congruence $\Cohto$ as $2$-cells.
\end{definition}

\begin{example}
  \label{ex:rs-mon}
  \label{rs-mon}
  The extended rewriting system $\Mon$ for monoids has symbols and rules
  \begin{align*}
    \Mon_1&=\set{m:2\to 1,e:0\to 1}\\
    \Mon_2&=\set{
      \begin{array}{r@{\ :\ }r@{\ \To\ }l}
        \alpha&m(m(x_1,x_2),x_3)&m(x_1,m(x_2,x_3))\\
        \lambda&m(e,x_1)&x_1\\
        \rho&m(x_1,e)&x_1            
      \end{array}
    }
  \end{align*}
  There are coherence relations~$A$, $B$, $C$, $D$ and $E$, respectively corresponding to a confluence for the five critical pairs of the rewriting system (as defined below), whose $0$-sources are respectively
  \begin{align*}
    m(m(m(x_1,x_2),x_3),x_4)
    &&
    m(m(e,x_1),x_2)
    &&
    m(m(x_1,e),x_2)
    &&
    m(m(x_1,x_2),e)
    &&
    m(e,e)
  \end{align*}
  Those coherence relations can be pictured as follows:
  \[
    \begin{tikzcd}[arrows={Rightarrow}]
      m(m(m(x_1,x_2),x_3),x_4)\ar[dd,"\alpha"']\ar[r,"\alpha"]&m(m(x_1,m(x_2,x_3)),x_4)\ar[dr,"\alpha"]\\
      \ar[rr,scaling nfold=3,"A",shorten=100pt]&&m(x_1,m(m(x_2,x_3),x_4))\ar[d,"\alpha"]\\
      m(m(x_1,x_2),m(x_3,x_4))\ar[rr,"\alpha"']&&m(x_1,m(x_2,m(x_3,x_4)))
    \end{tikzcd}
  \]
  \[
    \begin{array}{cc}
      \begin{tikzcd}[arrows={Rightarrow},column sep=-10pt]
        m(m(e,x_1),x_2)\ar[dr,"\lambda"']\ar[rr,"\alpha"]&\ar[d,phantom,"\overset B\TO",pos=.3]&\ar[dl,"\lambda"]m(e,m(x_1,x_2))\\
        &m(x_1,x_2)
      \end{tikzcd}
      &
      \begin{tikzcd}[arrows={Rightarrow},column sep=-10pt]
        m(m(x_1,e),x_2)\ar[dr,"\rho"']\ar[rr,"\alpha"]&\ar[d,phantom,"\overset C\TO",pos=.3]&\ar[dl,"\lambda"]m(x_1,m(e,x_2))\\
        &m(x_1,x_2)&
      \end{tikzcd}
      \\
      \begin{tikzcd}[arrows={Rightarrow},column sep=-10pt]
        m(m(x_1,x_2),e)\ar[dr,"\rho"']\ar[rr,"\alpha"]&\ar[d,phantom,"\overset D\TO",pos=.3]&m(x_1,m(x_2,e))\ar[dl,"\rho"]\\
        &m(x_1,x_2)
      \end{tikzcd}
      &
      \begin{tikzcd}[arrows={Rightarrow}]
        \ar[d,bend right,"\lambda"',""{name=l}]m(e,e)\ar[d,bend left,"\rho",""{name=r}]\ar[d,phantom,"\overset E\TO",pos=.3]\\
        e
      \end{tikzcd}
    \end{array}
  \]
  For concision, for each arrow, we did not indicate the proper rewriting step, but only the rewriting rule of the rewriting step (hopefully, the reader will easily be able to reconstruct the missing bicontext). For instance, the coherence relation~$C$ has type
  \[
    C
    :
    m(\rho(x_1),x_2)
    \To
    \alpha(x_1,e,x_2)\pcomp m(x_1,\lambda(x_2))
  \]
  so that the missing bicontexts for the rules labeled by $\rho$, $\alpha$ and $\lambda$ are respectively
  \begin{align*}
    m(\hole,x_2)[-\circ\tuple{x_1}]
    &&
    \hole[-\circ\tuple{x_1,e,x_2}]
    &&
    m(x_1,\hole)[-\circ\tuple{x_2}]
  \end{align*}
  This coherent term rewriting system has been considered in various places in literature~\cite{beke2011categorification,cohen2009coherence}.
\end{example}

\noindent
We mention here that the notion of Tietze transformation can be defined for \TRS2 in a similar way as for \ARS2 (\cref{ars2-tietze}):

\begin{definition}
  \label{trs2-tietze}
  The \emph{Tietze transformations} are the following possible transformations on a \TRS2~$\P$:
  \begin{enumerate}[(T2)]
  \item[(T1)] given a zig-zag $\pi:t\Zzto u$, add a new rewriting rule $\alpha:t\To u$ in $\P_2$ together with a new coherence relation $A:\alpha\TO\pi$ in~$\P_3$,
  \item[(T2)] given zig-zags $\pi,\rho:t\Zzto u$ such that $\pi\Cohto\rho$, add a new coherence relation $A:\pi\TO\rho$ in~$\P_3$.
  \end{enumerate}
  The \emph{Tietze equivalence} is the smallest equivalence relation on \TRS2 identifying $\P$ and $\Q$ whenever $\Q$ can be obtained from~$\P$ by a Tietze transformation (T1) or (T2).
\end{definition}

\begin{proposition}
  \label{trs2-tietze-correct}
  Any two Tietze equivalent \TRS2 present isomorphic groupoids.
\end{proposition}

\subsection{Rewriting properties}
\label{sec:trs-rewr}
Let~$\S$ be a \TRS{2} together with $W\subseteq\S_2$.
The \TRS{2}~$\S$ induces an \ARS{2} in each hom-set: this point of view will allow reusing the work done on \ARS{2} on \cref{sec:ars}.

\begin{definition}
  \label{hom-ars}
  Given a \TRS{2}~$\S$ and $n\in\N$, we write $\S(n,1)$ for the \ARS2 whose
  \begin{itemize}
  \item objects are the $n$-ary terms:
    \[
      \S(n,1)_0=\freecat\S_1(n,1)
    \]
  \item morphisms are the $n$-ary rewriting steps:
    \[
      \S(n,1)_1=\steps\S_2(n,1)
    \]
    where $\steps\S_2(n,1)$ is the set of rewriting steps
    \[
      C[\rho\circ f]:C[t\circ f]\To C[u\circ f]
    \]
    with both $C[t\circ f]$ and $C[u\circ f]$ of arity~$n$,
  \item coherence relations are triples $(C,A,f)$, written $C[A\circ f]$, for
    some context~$C$, coherence relation $A\in\S_3$ and substitution $f$, of
    suitable type, of the form
    \[
      C[A\circ f]
      :
      C[\pi\circ f]\TO C[\pi'\circ f]
      :
      C[t\circ f]\To C[u\circ f]
    \]
    such that both $C[t\circ f]$ and $C[u\circ f]$ of arity~$n$.
  \end{itemize}
\end{definition}

\noindent
Similarly, a set $W$ induces a set $W(m,1)\subseteq\S(m,1)_1$, where
$W(m,1)$ is the set of \emph{$W$-rewriting steps}, \ie rewriting steps of the
form $C[\alpha\circ f]$ with $\alpha\in W$.
We say that a \TRS{2}~$\S$ is \emph{$W$-terminating} / \emph{locally $W$-confluent} / \emph{$W$-confluent} / \emph{$W$-coherent} when each $\S(m,n)$ is with respect to $W(m,n)$. We say that~$\S$ is \emph{confluent} when it is $W$-confluent for~$W=\S_2$ (and similarly for other properties). More explicitly,

\begin{definition}
  A \emph{$W$-branching} $(\alpha_1,\alpha_2)$ is a pair of rewriting steps
  $\alpha_1:t\To u_1$ and $\alpha_2:t\To u_2$ in~$\steps\W$ with the same
  source:
  \[
    \begin{tikzcd}
      u_1&\ar[l,Rightarrow,"\alpha_1"']t\ar[r,Rightarrow,"\alpha_2"]&u_2
    \end{tikzcd}
  \]
  Such a $W$-branching is \emph{$W$-confluent} when there are cofinal rewriting
  paths $\pi_1:u_1\To v$ and $\pi_2:u_2\To v$ in $\freecat W$ such that
  $\pgpd{\alpha_1\pcomp\pi_1}=\pgpd{\alpha_2\pcomp\pi_2}$, which is depicted on
  the left
  
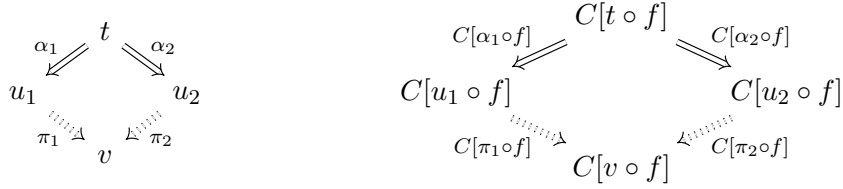
\[
    \begin{tikzcd}[sep=small]
      &\ar[dl,Rightarrow,"\alpha_1"']t\ar[dr,Rightarrow,"\alpha_2"]&\\
      u_1\ar[dr,Rightarrow,dotted,"\pi_1"']&&\ar[dl,Rightarrow,dotted,"\pi_2"]u_2\\
      &v
    \end{tikzcd}
    \qquad\qquad\qquad
    \begin{tikzcd}[sep=small]
      &\ar[dl,Rightarrow,"{C[\alpha_1\circ f]}"']C[t\circ f]\ar[dr,Rightarrow,"{C[\alpha_2\circ f]}"]&\\
      C[u_1\circ f]\ar[dr,Rightarrow,dotted,"{C[\pi_1\circ f]}"']&&\ar[dl,Rightarrow,dotted,"{C[\pi_2\circ f]}"]C[u_2\circ f]\\
      &C[v\circ f]
    \end{tikzcd}
  \]
\end{definition}

\noindent
By extension of \cref{prop:lc+n-c}, we have

\begin{proposition}
  \label{prop:lc+n-c2}
  If $\S$ is $W$-terminating and locally $W$-confluent then it is $W$-confluent.
\end{proposition}

\noindent
In practice, termination can be shown as
follows~\cite[Section 5.2]{baader1999term}.

\begin{definition}
  A \emph{reduction order}~$\geq$ is a well-founded preorder on terms in
  $\freecat\S_1$ which is compatible with context extension: given terms
  $t,u\in\freecat\S_1$, $t>u$ implies $C[t\circ f]>C[u\circ f]$ for every
  context $C$ and substitution~$f\in\freecat\S_1$ (whose types are such that the
  expressions make sense).
\end{definition}

\begin{proposition}
  \label{prop:ro-term}
  A \TRS{2}~$\S$ equipped with a reduction order such that $t>u$ for any rule
  $\alpha:t\To u$ in~$W$ is $W$-terminating.
\end{proposition}
\begin{proof}
  For any rewriting step $ C[\rho\circ f] : C[t\circ f] \To C[u\circ f] $ we
  have $C[t\circ f]>C[u\circ f]$ and we conclude by well-foundedness.
\end{proof}

\noindent
Moreover, in order to construct a reduction order one can use the following
``interpretation method''~\cite[Section~5.3]{baader1999term}.

\begin{proposition}
  Let $(X,\leq)$ be a well-founded poset together with an interpretation
  \[
    \intp{a}:X^n\to X
  \]
  of each symbol $a\in\S_1$ of arity~$n$ as a function which is strictly decreasing in each argument. This induces an interpretation $\intp{t}:X^n\to X$ of every term $t$ of arbitrary arity~$n$ defined by induction by
  \begin{align*}
    \intp{x^n_i}&=\pi_i^n
    &
    \intp{a(t_1,\ldots,t_n)}&=\intp{a}\circ\tuple{\intp{t_1},\ldots,\intp{t_n}}
  \end{align*}
  where $\pi^n_i:X^n\to X$ is the projection on the $i$-th coordinate.
  We define an order on functions $f,g:X^n\to X$ by
  \[
    f\succ g
    \qquad\text{iff}\qquad
    f(x_1,\ldots,x_n)\succ g(x_1,\ldots,x_n)
    \text{ for every $x_1,\ldots,x_n\in X$}
  \]
  and we still write $\succeq$ for the order on terms such that $t\succeq u$
  whenever $\intp{t}\succeq\intp{u}$.
  This order is always a reduction order.
\end{proposition}

\noindent
Note that given a reduction order $\succeq$ defined as above, by
\cref{prop:ro-term}, if we have $t\succ u$ for every rule $\alpha:t\To u$ the
\TRS{2} is $W$-terminating.

\begin{example}
  \label{ex:mon-term}
  \label{mon-term}
  Consider the \TRS2~$\Mon$ of \cref{ex:rs-mon}. We consider the set $X=\N\setminus{0}$ and interpret the symbols as
  \begin{align*}
    \intp{m(x_1,x_2)}&=2x_1+x_2
    &
    \intp{e}&=1    
  \end{align*}
  All the rules are decreasing since we have
  \begin{align*}
    \intp{m(m(x_1,x_2),x_3)}=4x_1+2x_2+x_3&>2x_1+2x_2+x_3=\intp{m(x_1,m(x_2,x_3))}
    \\
    \intp{m(e,x_1)}=2+x_1&>x_1=\intp{x_1}
    \\
    \intp{m(x_1,e)}=2x_1+1&>x_1=\intp{x_1}    
  \end{align*}
  and the rewriting system is terminating.
\end{example}

We now briefly recall the notion of \emph{critical pair},
see~\cite{malbos2016homological} for a more detailed presentation. We say that a
branching $(\alpha_1,\alpha_2)$ is \emph{smaller} than a branching
$(\beta_1,\beta_2)$ when the second can be obtained from the first by
``extending the context'', \ie when there exists a context~$C$ and a
morphism~$f$ of suitable types such that $\beta_i=C[\alpha_i\circ f]$ for
$i=1,2$. In this case, the confluence of the first branching implies the
confluence of the second one (see the diagram on the right above). The notion of
context can be generalized to define the notion of a binary context~$C$, with
two holes, each of which occurs exactly once: we write $C[t,u]$ for the context
where the holes have respectively been substituted with terms~$t$ and~$u$. A
branching is \emph{orthogonal} when it consists of two rewriting steps at
disjoint positions, \ie when it is of the form
\[
  \begin{tikzcd}[sep=15ex]
    C[u_1\circ f_1,t_2\circ f_2]&\ar[l,Rightarrow,"{C[\alpha_1\circ f_1,t_2\circ f_2]}"']C[t_1\circ f_1,t_2\circ f_2]\ar[r,Rightarrow,"{C[t_1\circ f_1,\alpha_2\circ f_2]}"]&C[t_1\circ f_1,u_2\circ f_2]
  \end{tikzcd}
\]
for some binary context $C$, rewriting rules $\alpha_i:t_i\To u_i$ in $\S_2$ and
morphisms $f_i$ in $\freecat\S_1$ of suitable types. A branching forms a
\emph{critical pair}\label{critical-pair} when it is not orthogonal and minimal (\wrt the above
order). A \TRS{} with a finite number of rewriting rules always have a finite
number of critical pairs and those can be computed
efficiently~\cite{baader1999term}.

\begin{lemma}
  \label{lem:trs-cb-lc}
  A \TRS{2}~$\S$ is locally $W$-confluent when all its critical $W$-branchings
  are $W$-confluent.
\end{lemma}
\begin{proof}
  Suppose that all critical $W$-branchings are confluent. A non-overlapping
  $W$-branching is easily shown to be $W$-confluent. A non-minimal $W$-branching
  is greater than a minimal one, which is $W$-confluent by hypothesis, and is
  thus itself also $W$-confluent.
\end{proof}

\noindent
We write $W_3\subseteq\S_3$ for the set of coherence relations $A:\pi\To\rho$
such that both $\pi$ and $\rho$ belong to $\freegpd W$. As a useful particular
case, we have the following variant of the Squier theorem:

\begin{lemma}
  If \TRS{2}~$\S$ has a coherence relation in~$W_3$ corresponding to a choice of
  confluence for every critical $W$-branching then it is locally $W$-confluent.
\end{lemma}

\begin{example}
  \label{ex:mon-confl}
  The \TRS{2}~$\Mon$ of \cref{ex:rs-mon}. By definition, every critical pair is confluent and $\Mon$ is thus locally confluent. From \cref{ex:mon-term,prop:lc+n-c2}, we deduce that it is confluent.
\end{example}

\noindent
As a direct consequence of \cref{prop:ars-cr,coherent-rigid}, we have

\begin{proposition}
  \label{lem:tc-coh}
  \label{2-coh-rigid}
  If~$\S$ is $W$-terminating and locally $W$-confluent then it is $W$-coherent, and the subgroupoid~$\W$ of $\pgpd\S$ generated by~$W$ is thus rigid.
\end{proposition}

\noindent
For instance, from \cref{ex:mon-term,ex:mon-confl}, we deduce that the \TRS2~$\Mon$ is coherent, thus showing the coherence property (C1) for monoidal categories.

%
Fix a $W$-convergent \TRS{2}~$\S$. By \cref{lem:tc-coh}, $\pgpd\S$ is $W$-coherent, by \cref{prop:rigid-2equivalence}, the quotient functor $\pgpd\S\to\pgpd\S/W$ is a local equivalence, and by \cref{prop:quot-nf}, $\pgpd\S/W$ is obtained from $\pgpd\P$ by restricting to $1$-cells in normal form. Moreover, in good situations, we can provide a description of the quotient category~$\pgpd\S/W$ by applying \cref{prop:ars-nf-pres} hom-wise.


\section{Coherence for symmetric monoidal categories}
\label{smc-coh}
In this section, we illustrate the use of the methods developed in the article, by applying them in order to recover the coherence theorems for symmetric monoidal categories~\cite{joyal1993braided}, which requires quotienting the Lawvere 2-theory of symmetric monoids by a subtheory~$\W$, generated by the associator and the unitors. Related results using rewriting in polygraphs where obtained earlier~\cite{ acclavio2016constructive,guiraud2012coherence,lafont2003towards}. They however require heavier computations since manipulations of variables (duplication, erasure and commutation) need to be implemented as explicit rules in this context.

\subsection{A theory for symmetric monoidal categories}
A \emph{symmetric monoidal category} is a monoidal category equipped with a natural isomorphism $\gamma_{x,y}:x\otimes y\to y\otimes x$, called the \emph{symmetry}, satisfying the three axioms recalled in \cref{smc-def}. A symmetric monoidal category is \emph{strict} when the structural isomorphisms $\alpha$, $\lambda$ and $\rho$ are identities (but we do not require $\gamma$ to be an identity). We write $\SMonCat$ (\resp $\StrSMonCat$) for the category of symmetric monoidal categories (\resp strict ones).

We write $\SMon$ for the \TRS{2} obtained from $\Mon$, see \cref{ex:rs-mon}, by adding a rewriting rule
\[
  \gamma:m(x_1,x_2)\To m(x_2,x_1)
\]
corresponding to symmetry, together with a coherence relation
\[
  F:\gamma(x_1,x_2)\pcomp\gamma(x_2,x_1)\TO\id_{m(x_1,x_2)}
\]
which can be pictured as
\[
  \begin{tikzcd}
  m(x_1,x_2)\ar[d,equals]\ar[r,Rightarrow,"\gamma"]\ar[dr,phantom,"\overset F\TO"]&m(x_2,x_1)\ar[d,Rightarrow,"\gamma"]\\
  m(x_1,x_2)\ar[r,equals]&m(x_1,x_2)
\end{tikzcd}
\]
as well as the coherence relations
\[
  \begin{tikzcd}[arrows={Rightarrow},column sep=small]
    m(m(x_1,x_2),x_3)\ar[d,"\alpha"']\ar[r,"\gamma"]\ar[drr,phantom,"\overset G\TO"]&m(m(x_2,x_1),x_3)\ar[r,"\alpha"]&m(x_2,m(x_1,x_3))\ar[d,"\gamma"]\\
    m(x_1,m(x_2,x_3))\ar[r,"\gamma"']&m(m(x_2,x_3),x_1)\ar[r,"\alpha"']&m(x_2,m(x_3,x_1))
  \end{tikzcd}
  \qquad
  \begin{tikzcd}[arrows={Rightarrow},column sep=0pt]
    m(e,x_1)\ar[dr,"\lambda"']\ar[rr,"\gamma"]&\ar[d,phantom,pos=.3,"\overset I\TO"]&\ar[dl,"\rho"]m(x_1,e)\\
    &x_1
  \end{tikzcd}
\]
It is immediate to see that the algebras of $\SMon$ are precisely symmetric
monoidal categories:

\begin{proposition}
  \label{smoncat-alg}
  The category $\Alg(\pgpd\SMon)$ is isomorphic to the category $\SMonCat$.
\end{proposition}

Since our aim is to study the relationship between symmetric monoidal categories and their strict version, it is natural to consider the set of rewriting rules
\[
  W=\set{\alpha,\lambda,\rho}
\]
\ie all the rules excepting~$\gamma$. Namely,

\begin{lemma}
  The category $\Alg(\SMon/W)$ is isomorphic to the category $\StrSMonCat$ of
  strict monoidal categories.
\end{lemma}

\begin{lemma}
  The \TRS2 $\SMon$ is $W$-coherent.
\end{lemma}
\begin{proof}
  Since $W$ consists in $\alpha$, $\lambda$ and $\rho$ only, this can be deduced as in the case of monoids: the \TRS2 is $W$-terminating by \cref{ex:mon-term} and $W$-locally confluent by definition (\cref{ex:rs-mon}), it is thus $W$-coherent by \cref{lem:tc-coh}.
\end{proof}

Provided that \cref{conj:strict-equiv} holds, we could deduce that any symmetric
monoidal category is monoidally equivalent to a strict one.
Note that the above reasoning only depends on the convergence of the subsystem
induced by~$W$, \ie on the fact that every diagram made of~$\alpha$, $\lambda$
and~$\rho$ commutes, but it does not require anything on diagrams
containing~$\gamma$'s. In particular, if we removed the compatibility relations
$G$, $H$, $I$ and $J$, the strictification theorem would still hold. The
resulting notion of strict symmetric monoidal category would however be worrying
since, for instance, in absence of~$I$, the 2-cell
\[
  \gamma_{e,x_1}:m(e,x_1)\To m(x_1,e)
\]
would induce, in the quotient, a non-trivial automorphism
\[
  \gamma_{e,x_1}:x_1\To x_1
\]
of each variable~$x_1$. We prove below (\cref{smoncat-strong-coh}) a variant of the coherence theorem that is ``stronger'' in the sense that it requires these axioms to hold and implies that the identity is the only automorphism of~$x_1$.


\subsection{Every affine diagram commutes}
We have seen that, for the theory of monoidal categories, ``every diagram commutes'', in the sense that $\pgpd\Mon$ is a $2$-rigid $(2,1)$-category. For symmetric monoidal categories, we do not expect this to hold since we have two rewriting paths
\begin{align*}
  \gamma_{x_1,x_1}:m(x_1,x_1)&\To m(x_1,x_1)
  &
  \id_{m(x_1,x_1)}:m(x_1,x_1)&\To m(x_1,x_1)
\end{align*}
which are both from $m(x_1,x_1)$ to itself, and are not equal in general as explained in the introduction. It can however be shown that it holds for the subclass of $2$-cells whose source and target are affine terms:

\begin{definition}
  A term~$t$ is \emph{affine} if no variable occurs twice, \ie $\occurrences it\leq 1$ for every index~$i$.
\end{definition}

\noindent
We now explain this, thus recovering a well-known property~\cite[Theorem 4.1]{maclane1963natural} using rewriting techniques. In order to use those, it will be convenient to work with a variant $\SMon'$ of the \TRS2 $\SMon$, obtained by adding a new generator~$\delta$ as well as coherence relations corresponding to the local confluence diagrams: this variant will allow proving \cref{P-P'} below. A similar completion has been investigated by Lafont~\cite[Figure~9]{lafont2003towards}, although working in a monoidal setting whereas we are working in a cartesian one.
We write~$\SMon'$ for the \TRS2 obtained from $\SMon$ by adding a rewriting rule
\[
  \delta:m(x_1,m(x_2,x_3))\To m(x_2,m(x_1,x_3))
\]
removing the coherence relation~$G$ and adding coherence relations
\[
  \begin{tikzcd}[arrows={Rightarrow}]
    m(x_1,m(x_2,x_3))\ar[d,equals]\ar[dr,phantom,"\overset{F'}\TO"]\ar[r,"\delta"]&m(x_2,m(x_1,x_3))\ar[d,"\delta"]\\
    m(x_1,m(x_2,x_3))\ar[r,equals]&m(x_1,m(x_2,x_3))
  \end{tikzcd}
  \qquad
  \begin{tikzcd}[arrows={Rightarrow}]
    m(m(x_1,x_2),x_3)\ar[d,"\alpha"']\ar[dr,phantom,"\overset{G'}\TO"]\ar[r,"\gamma"]&m(m(x_2,x_1),x_3)\ar[d,"\alpha"]\\
    m(x_1,m(x_2,x_3))\ar[r,"\delta"']&m(x_2,m(x_1,x_3))
  \end{tikzcd}
\]
\[
  \begin{tikzcd}[arrows={Rightarrow},column sep=small]
    m(m(x_1,x_2),x_3)\ar[d,"\alpha"']\ar[rr,"\gamma"]\ar[drr,phantom,"\overset H\TO"]&&m(x_3,m(x_1,x_2))\ar[d,equals]\\
    m(x_1,m(x_2,x_3))\ar[r,"\gamma"']&m(x_1,m(x_3,x_2))\ar[r,"\delta"']&m(x_3,m(x_1,x_2))
  \end{tikzcd}
  \quad
  \begin{tikzcd}[arrows={Rightarrow},column sep=0pt]
    m(x_1,e)\ar[dr,"\rho"']\ar[rr,"\gamma"]&\ar[d,phantom,pos=.3,"\overset J\TO"]&\ar[dl,"\lambda"]m(e,x_1)\\
    &x_1
  \end{tikzcd}
\]
\[
  \begin{tikzcd}[arrows={Rightarrow}]
    m(m(x_1,x_2),m(x_3,x_4))\ar[drr,phantom,"\overset K\TO"]\ar[d,"\alpha"']\ar[rr,"\delta"]&&m(x_3,m(m(x_1,x_2),x_4))\ar[d,"\alpha"]\\
    m(x_1,m(x_2,m(x_3,x_4)))\ar[r,"\delta"']&m(x_1,m(x_3,m(x_2,x_4)))\ar[r,"\delta"']&m(x_3,m(x_1,m(x_2,x_4)))
  \end{tikzcd}
\]
\[
  \begin{tikzcd}[arrows={Rightarrow}]
    m(x_1,m(m(x_2,x_3),x_4))\ar[drr,phantom,"\overset L\TO"]\ar[d,"\alpha"']\ar[rr,"\delta"]&&m(m(x_2,x_3),(x_1,x_4))\ar[d,"\alpha"]\\
    m(x_1,m(x_2,m(x_3,x_4)))\ar[r,"\delta"']&m(x_2,m(x_1,m(x_3,x_4)))\ar[r,"\delta"']&m(x_2,m(x_3,m(x_1,x_4)))
  \end{tikzcd}
\]
\[
  \begin{tikzcd}[arrows={Rightarrow}]
    m(x_1,m(x_2,x_3))\ar[dr,phantom,"\overset{M}\TO"]\ar[d,"\gamma"']\ar[r,"\delta"]&m(x_2,m(x_1,x_3))\ar[d,"\gamma"]\\
    m(m(x_2,x_3),x_1)\ar[r,"\alpha"']&m(x_2,m(x_3,x_1))
  \end{tikzcd}
\]
\[
  \begin{tikzcd}[arrows={Rightarrow}]
    m(x_1,m(x_2,x_3))\ar[drr,phantom,"\overset{N}\TO"]\ar[d,"\gamma"']\ar[r,"\delta"]&m(x_2,m(x_1,x_3))\ar[r,"\delta"]&m(x_2,m(x_3,x_1))\ar[d,"\delta"]\\
    m(x_1,m(x_3,x_2))\ar[r,"\delta"']&m(x_3,m(x_1,x_2))\ar[r,"\gamma"']&m(x_3,m(x_2,x_1))
  \end{tikzcd}
\]
\[
  \begin{tikzcd}[arrows={Rightarrow}]
    m(x_1,m(x_2,m(x_3,x_4)))\ar[drr,phantom,"\overset{O}\TO"]\ar[d,"\delta"']\ar[r,"\delta"]&m(x_2,m(x_1,m(x_3,x_4)))\ar[r,"\delta"]&m(x_2,m(x_3,m(x_1,x_4)))\ar[d,"\delta"]\\
    m(x_1,m(x_3,m(x_2,x_4)))\ar[r,"\delta"']&m(x_3,m(x_1,m(x_2,x_4)))\ar[r,"\delta"']&m(x_3,m(x_2,m(x_1,x_4)))
  \end{tikzcd}
\]
\[
  \begin{tikzcd}[arrows={Rightarrow},column sep=-5pt]
    m(e,m(x_1,x_2))\ar[dr,"\lambda"']\ar[rr,"\delta"]&\ar[d,phantom,"\overset{P}\TO"]&m(x_1,m(e,x_2))\ar[dl,"\lambda"]\\
    &m(x_1,x_2)
  \end{tikzcd}
  \quad
  \begin{tikzcd}[arrows={Rightarrow},column sep=-5pt]
    m(x_1,m(e,x_2))\ar[dr,"\lambda"']\ar[rr,"\delta"]&\ar[d,phantom,"\overset{Q}\TO"]&m(e,m(x_1,x_2))\ar[dl,"\lambda"]\\
    &m(x_1,x_2)
  \end{tikzcd}
\]
\[
  \begin{tikzcd}[arrows={Rightarrow},column sep=0pt]
    m(x_1,m(x_2,e))\ar[dr,"\rho"']\ar[rr,"\delta"]&\ar[d,phantom,"\overset{R}\TO"]&m(x_2,m(x_1,e))\ar[dl,"\rho"]\\
    &m(x_1,x_2)
  \end{tikzcd}
\]

\noindent
The following adapts~\cite[Proposition~3.3.5]{guiraud2012coherence} from the monoidal setting to our cartesian setting, see also~\cite{acclavio2016constructive}:

\begin{proposition}
  \label{smon-completion}
  The \TRS2 $\SMon$ and $\SMon'$ present isomorphic categories.
\end{proposition}
\begin{proof}
  By \cref{trs2-tietze-correct}, it is enough to show that both \TRS2 are Tietze
  equivalent. The commutation of~$H$ and~$I$ is immediate in presence of the
  other axioms, we can thus add them using Tietze transformations~(T2).
  Namely, we can show the commutation of~$H$ by using~$F$ and~$G$ twice:
  \[
     \begin{tikzcd}[arrows={Rightarrow}]
       m(m(x_1,x_2),x_3)\ar[dr,Rightarrow,scaling nfold=3,"G"{xshift=-6pt, yshift=1pt},shorten=20pt]\ar[dd,"\alpha"']\ar[rrr,bend left=15,"\gamma"{name=g1}]&&&\ar[lll,"\gamma"description,""{name=g2}]m(x_3,m(x_1,x_2))\ar[dd,"\gamma"]\\
       &m(m(x_1,x_3),x_2)\ar[d,"\alpha"]\ar[r,bend right=15,"\gamma"',""'{name=g4}]&\ar[l,bend right=15,"\gamma"',""'{name=g3}]m(m(x_3,x_1),x_2)\ar[ur,"\alpha"]\ar[dr,Rightarrow,scaling nfold=3,"G"{xshift=-6pt, yshift=1pt},shorten=20pt]\\
       m(x_1,m(x_2,x_3))\ar[r,"\gamma"']&m(x_1,m(x_3,x_2))\ar[r,"\gamma"']&m(m(x_3,x_2),x_1)\ar[r,"\alpha"']&m(x_3,m(x_2,x_1))
       \ar[from=g1,to=g2,Rightarrow,scaling nfold=3,shorten=8pt,"F"{', xshift=-3pt}]
       \ar[from=g3,to=g4,Rightarrow,scaling nfold=3,shorten=6pt,"F"{', xshift=-3pt}]
  \end{tikzcd}
  \]
  and the commutation of~$J$ can be obtained from~$F$ and $I$:
  \[
    \begin{tikzcd}[arrows={Rightarrow}]
      m(x_1,e)\ar[dr,"\rho"']\ar[rr,bend left,"\gamma"{name=g1}]&\ar[d,phantom,pos=.35,"\overset I\Lleftarrow"]&\ar[ll,"\gamma"description,""{name=g2}]\ar[dl,"\lambda"]m(e,x_1)\\
      &x_1
      \ar[from=g1,to=g2,Rightarrow,scaling nfold=3,shorten=5pt,"F"{', xshift=-3pt}]
    \end{tikzcd}
  \]
  Next, by a Tietze transformation~(T1), we can add the rule $\delta$ together with its definition
  \[
    \delta(x_1,x_2,x_3)
    =
    \alpha(x_1,x_2,x_3)\circ m(\gamma(x_1,x_2),x_3)\circ\alpha(x_1,x_2,x_3)^{-1}
  \]
  which is formally given by the relation~$G'$.
  From this definition, one easily shows that the coherence relations~$K$ to~$R$ are derivable and can thus be added by Tietze transformations~(T2).
  Finally, the coherence relation~$G$ is then superfluous, since it can be derived as
  \[
    \begin{tikzcd}[arrows={Rightarrow},column sep=small]
      m(m(x_1,x_2),x_3)\ar[d,"\alpha"']\ar[r,"\gamma"]\ar[dr,phantom,"\overset{G'}\TO",pos=.3]&m(m(x_2,x_1),x_3)\ar[r,"\alpha"]\ar[dr,phantom,"\overset{M}\TO",pos=.7]&m(x_2,m(x_1,x_3))\ar[d,"\gamma"]\\
      m(x_1,m(x_2,x_3))\ar[r,"\gamma"']\ar[urr,"\delta"]&m(m(x_2,x_3),x_1)\ar[r,"\alpha"']&m(x_2,m(x_3,x_1))
    \end{tikzcd}
  \]
  and can thus be removed by a Tietze transformation~(T2).
\end{proof}

\begin{lemma}
  \label{smon'-lc}
  The \TRS{2} $\SMon'$ is locally confluent.
\end{lemma}
\begin{proof}
  By \cref{lem:trs-cb-lc}, it is enough to show that all the critical pairs are confluent, which holds by definition of $\SMon'$: the critical pairs involving $\alpha$, $\lambda$ and $\rho$ are handled in \cref{ex:rs-mon}, and those involving $\gamma$ or $\delta$ and another rewriting rule in the above definition of~$\SMon'$.
\end{proof}

\noindent
Again, analogous results in a monoidal (as opposed to cartesian) setting predate this work, see~\cite[Figure~51]{lafont2003towards} and \cite[Section~5.3.3]{guiraud-phd}.

The \TRS2 $\SMon'$ is not terminating (even when restricted to affine terms)
because of the rules $\gamma$ and~$\delta$ which witnesses for the commutativity
of the operation~$m$: for instance, we have the loop
\begin{equation}
  \label{gamma-loop}
  \begin{tikzcd}[arrows={Rightarrow},sep=large]
    m(x_1,x_2)\ar[r,"{\gamma(x_1,x_2)}"]&m(x_2,x_1)\ar[r,"{\gamma(x_2,x_1)}"]&m(x_1,x_2)
  \end{tikzcd}
\end{equation}
In order to circumvent this problem, we are going to formally ``remove'' the
second morphism above and only keep instances of~$\gamma$ (\resp $\delta$) which
tend to make variables in decreasing order. Namely, by the coherence
relation~$F$, \ie
\[
  \begin{tikzcd}[arrows={Rightarrow}]
    &m(x_2,x_1)\ar[d,Rightarrow,"F"{', xshift=-3pt},scaling nfold=3,shorten=6pt]\ar[dr,"{\gamma(x_2,x_1)}"]&\\
    m(x_1,x_2)\ar[ur,"{\gamma(x_1,x_2)}"]\ar[rr,equals]&{}&m(x_1,x_2)
  \end{tikzcd}
\]
we have $\gamma(x_2,x_1)=\gamma(x_1,x_2)^-$ so that $\gamma(x_2,x_1)$ is superfluous and we can remove it, by using Tietze transformations, without changing the presented $(2,1)$-category. Note that this operation is clearly not stable under substitution (for instance consider the substitution $\tuple{x_2,x_1}$ which exchanges the two variable names), so that this cannot actually be performed at the level of $(2,1)$-categories, but it can if we work within the hom-groupoids, which will be enough for our purposes.
If we remove all the rewriting steps involving $\gamma$ which tend to put variables in increasing order as explained above, we still have some loops such as
\[
  \begin{tikzcd}[arrows={Rightarrow}]
    m(e,x_1)\ar[r,"{\gamma(e,x_1)}"]&m(x_1,e)\ar[r,"{\gamma(x_1,e)}"]&m(e,x_1)
  \end{tikzcd}
\]
Intuitively, this is because the above rewriting path involves terms containing a unit~$e$, whereas our previous criterion relies on the order of variables. Fortunately, we can first remove all units by applying the rules~$\lambda$ and~$\rho$, and then apply the above argument.

Fix an arity $n\in\N$, and consider the \ARS2~$\P=\SMon'(n,1)$ as defined in \cref{hom-ars}. We write $\P'$ for the \ARS2 obtained from~$\P$ by
\begin{itemize}
\item removing from $\P_1$ the terms where the unit $e$ occurs, excepting~$e$ itself,
\item removing from $\P_2$ the rewriting steps whose source or target terms contain~$e$ (in particular, we remove all rewriting steps involving~$\lambda$
  or~$\rho$),
\item removing from~$\P_3$ the coherence relations where a removed step occurs in the source or the target.
\end{itemize}

\begin{lemma}
  \label{P-P'}
  The groupoid presented by~$\P'$ is equivalent to the one presented by~$\P$.
\end{lemma}
\begin{proof}
  We write $W\subseteq\P_1$ for the set of rewriting steps involving~$\lambda$
  or~$\rho$.
  By \cref{smon'-lc}, the \ARS2 $\P'$ is locally $W$-confluent, and thus, by \cref{normal-subcat-equiv}, $\pgpd\P$ is equivalent to $N(\pgpd\P)$, the full subcategory on $W$-normal forms. In turn, by \cref{prop:ars-nf-pres} (see below for details), the category $N(\pgpd\P)$ is isomorphic to the groupoid presented by~$\P\setminus W$ and we conclude by observing that the \ARS2 $\P'$ is precisely~$\P\setminus W$ (terms in normal form are precisely those where $e$ does not occur, with the exception of $e$ itself).

  Let us explain why the conditions of \cref{prop:ars-nf-pres} are satisfied.
  \begin{enumerate}
  \item We have seen above that $\P$ is $W$-convergent.
  \item It is immediate to check that no rewriting rule can produce a term containing~$e$ from a term which does not have this property.
  \item From \cref{ars-nf-pres-local}, in order to show this condition, we have to check that every diagram of the form
    \[
      \begin{tikzcd}
        t\ar[d,"\omega"']\ar[r,"\alpha"]&u\\
        t'
      \end{tikzcd}
    \]
    can be closed as in \cref{eq:ars-nf-pres-res} for arbitrary rewriting steps $\alpha\in\P_1$ and $\omega\in W$. It is enough to show this when they form a critical pair. There are five of them, which correspond to the five coherence relations~$B$, $C$, $D$, $I$ and~$J$, from which we conclude.
  \item From \cref{ars-nf-pres-local}, in order to show this condition, we have to check that every diagram of the form
    \[
      \begin{tikzcd}
        t\ar[d,"\omega"']\ar[r,bend left,"\alpha",""{name=alpha}]\ar[r,bend right,"\beta"',""'{name=beta}]&u\\
        t'
        \ar[from=alpha,to=beta,Rightarrow,"X"',shorten=4pt]
      \end{tikzcd}
    \]
    can be closed as in \cref{eq:ars-nf-pres-res-coh} for $\omega$ in~$W$. Again, it is enough to show this in situations which are not orthogonal and minimal, in a similar sense as for critical pairs,
    which we call a ``critical pair between a rewriting rule and a coherence relation''. For instance, one such critical pair between~$\lambda$ and~$A$ can be closed as follows:
    \[
      \begin{tikzcd}
        m(m(m(e,x_2),x_3),x_4)\ar[dd,bend right,dotted]\ar[d,"\alpha"'{name=a1}]\ar[r,"\alpha"]&\ar[dd,bend right,dotted]m(m(e,m(x_2,x_3)),x_4)\ar[r,"\alpha"]&m(e,m(m(x_2,x_3),x_4))\ar[dd,bend left,dotted]\ar[d,"\alpha"{name=a2}]\\
        m(m(e,x_2),m(x_3,x_4))\ar[dd,bend right,dotted]\ar[rr,"\alpha"']&&m(e,m(x_2,m(x_3,x_4)))\ar[dd,bend left,dotted]\\
        m(m(x_2,x_3),x_4)\ar[d,"\alpha"'{name=a3}]\ar[r,equals]&m(m(x_2,x_3),x_4)\ar[r,equals]&m(m(x_2,x_3),x_4)\ar[d,"\alpha"{name=a4}]\\
        m(x_2,m(x_3,x_4))\ar[rr,equals]&&m(x_2,m(x_3,x_4))
        \ar[from=a1,to=a2,Rightarrow,"A",shorten=120pt]
        \ar[from=a3,to=a4,equals,shorten=120pt]
      \end{tikzcd}
    \]
    where the vertical dotted arrows are the obvious rewriting steps involving~$\lambda$ (this is the only critical pair between~$\lambda$ and~$A$ and there are three critical pairs between $\rho$ and~$A$).\qedhere
  \end{enumerate}
\end{proof}

\begin{lemma}
  \label{P'-lc}
  The \ARS2~$\P'$ is locally confluent.
\end{lemma}
\begin{proof}
  We can deduce local confluence of~$\P'$ from the one of~$\P$: given a local
  pair with~$t$ as source, it is confluent in~$\P$ by \cref{smon'-lc} and
  thus in~$\P'$. Namely, since~$t$ lies in~$\P'$, the whole diagram does by
  property \cref{ars-nf-pres-nt} shown in the proof of \cref{P-P'} above.
\end{proof}

\noindent
Given a term~$t$, we write $\vars t$ the list of variables occurring in it, from
left to right, \eg $\vars{m(m(x_2,e),x_1)}=x_2x_1$. We order variables by
$x_i\succeq x_j$ whenever $i\leq j$ and extend it to lists of variables by
lexicographic ordering. Given terms $t$ and $u$, we write $t\succeq u$ when
$\vars t$ is greater than~$u$ according to the preceding order.

\begin{lemma}
  \label{affine-wf}
  The preorder $\succeq$ is well-founded on affine terms with fixed arity.
\end{lemma}
\begin{proof}
  Any infinite decreasing sequence $t_1\succ t_2\succ\ldots$ of terms, would
  induce an infinite decreasing sequence $\vars t_1\succ\vars t_2\succ\ldots$ of
  lists of variables, but there is only a finite number of those since we
  consider affine terms (so that there are no repetitions of variables) of fixed
  arity (so that there is a finite number of variables).
\end{proof}


%
A rewriting step $\rho:t\To u$ in $\P_1'$ is \emph{decreasing} when $t\succ u$.
We write~$\P''$ for the \ARS2 obtained from~$\P'$ by
\begin{itemize}
\item removing from~$\P_1'$ all the rewriting steps of the form
  \[
    C[\gamma(t_1,t_2)]
    :
    C[m(t_1,t_2)]
    \To
    C[m(t_2,t_1)]
  \]
  which are not decreasing,
\item replacing in the source or target of a relation in~$\P_2'$ all the
  non-decreasing steps $C[\gamma(t_1,t_2)]$ by $C[\gamma(t_2,t_1)^-]$.
\end{itemize}

\begin{lemma}
  \label{P'-P''}
  The \ARS2 $\P'$ and $\P''$ present isomorphic groupoids.
\end{lemma}
\begin{proof}
  This is a direct application of \cref{tietze-remove-rule}, because $\P_3'$
  contains the coherence rules
  \[
    \begin{tikzcd}[row sep=small]
      &C[m(t_2,t_1)\circ f]\ar[dr,bend left=10,"{C[\gamma(t_2,t_1)\circ f]}"]\ar[phantom,""{name=top}]\\
      C[m(t_1,t_2)\circ f]\ar[ur,bend left=10,"{C[\gamma(t_1,t_2)\circ f]}"]\ar[rr,equals,bend right=10,""'{name=eq}]&{}&C[m(t_1,t_2)\circ f]
      \ar[from=top,to=eq,Rightarrow,"{C[F\circ f]}"',shorten=10pt]
    \end{tikzcd}
  \]
  which allow us to conclude.
\end{proof}

In the following, for a \ARS2 whose objects are terms such as $\P''$, we say that it is ``terminating on affine terms'' when there is no infinite sequence of rewriting steps $t_0\to t_1\to\ldots$ where the term $t_0$ is affine. Note that all the rules of~$\P''$ rewrite affine terms into affine terms, so that all the $t_i$ are also necessarily affine in such a sequence of rewriting steps.

\begin{lemma}
  \label{P''-confl}
  The \ARS{2}~$\P''$ is terminating on affine terms and locally confluent.
\end{lemma}
\begin{proof}
  %
  In order to show termination, we can take the lexicographic product of the
  orders~$\succeq$ of \cref{affine-wf} and the one of \cref{mon-term}. This
  order is well-founded as a lexicographic product of well-founded orders. The
  rewriting steps involving $\gamma$ are strictly decreasing \wrt $\succeq$, by
  definition of $\P''$. The rewriting steps involving~$\alpha$ are left
  invariant by~$\succeq$ but are strictly decreasing \wrt the second order. We
  deduce that~$\P''$ is terminating.

  We have seen in \cref{P'-lc} that $\P'$ is locally confluent. We thus have
  that $\P''$ is also locally confluent because the confluence diagrams
  involving $\gamma$ (namely $G$, $H$, $I$ and $J$) only require decreasing
  instances of rewriting rules involving $\gamma$.
\end{proof}

\noindent
As a direct consequence, we have:

\begin{lemma}
  \label{affine-coh}
  Given two rewriting zig-zags $p,q:t\zzto u$ in $\P''$ such that $t$ is affine, we have that $u$ is also affine and $p\cohto q$.
\end{lemma}
\begin{proof}
  By \cref{P''-confl}, the restriction of $\P''$ to affine terms is terminating and locally confluent, thus confluent by \cref{newman} and thus coherent by \cref{prop:ars-cr}.
\end{proof}

\noindent
From the properties shown in \cref{ars-rewriting}, we deduce that $\P''$ is coherent, which implies the following:

\begin{lemma}
  \label{smon-affine-coh}
  Given two terms $t,u:n\to 1$, with $t$ affine, there is at most one 2-cell $t\To u$ in $\pgpd\SMon$.
\end{lemma}
\begin{proof}
  We write $\P=\SMon'(n,1)$. Given two rewriting zig-zags $p,q:t\Zzto u$, we have
  \begin{align*}
    \pgpd{\SMon}(p,q)
    &\isoto\pgpd\P(p,q)&&\text{by \cref{smon-completion}}\\
    &\isoto\pgpd\P'(p,q)&&\text{by \cref{P-P'},}\\
    &\isoto\pgpd\P''(p,q)&&\text{by \cref{P'-P''},}\\
    &\isoto 1&&\text{by \cref{affine-coh},}
  \end{align*}
  from which we conclude.
\end{proof}

\noindent
Finally, we thus conclude to the following coherence theorem~(S1):

\begin{theorem}
  \label{smoncat-strong-coh}
  \label{smc-S1}
  In a symmetric monoidal category, every diagram whose $0$-source is a tensor product of distinct objects, and whose morphisms are composites and tensor products of structural morphisms, commutes.
\end{theorem}
\begin{proof}
  Fix a symmetric monoidal category~$\C$. By \cref{smoncat-alg}, $\C$ can be seen as a product preserving 2-functor $\pgpd\SMon\to\Cat$. A coherence diagram in~$\C$ thus corresponds to a pair of rewriting paths $p,q:t\Pathto u$ in~$\freecat\SMon_2$ for some terms~$t$ and~$u$. When the 0-source of the coherence diagram is a tensor product of distinct objects, $t$ affine, and we have $p=q$ in $\pgpd\SMon$ by \cref{smon-affine-coh}.
\end{proof}

Previous theorem essentially follows from \cref{affine-coh}, which identifies a class of 1-cells of~$\pgpd\SMon$, the affine ones, between which there is at most one 2-cell. The following proposition characterizes, among them, those between which there actually exists a 2-cell.

\begin{proposition}
  \label{affine-2-cells}
  Given two affine terms~$t$ and~$u$, there exists a rewriting zig-zag $p:t\Zzto u$ if and only if~$t$ and~$u$ have the same variables.
\end{proposition}
\begin{proof}
  The left-to-right implication can be proved by checking that all the rewriting rules of~$\SMon$ preserve the variables of rewritten terms.
  Now, consider the right-to-left implication and suppose that $t$ and~$u$ are two terms with the same variables. By \cref{ex:mon-term,normal-subcat-equiv}, we can suppose that~$t$ and~$u$ are in normal form with respect to the rules $W=\set{\alpha,\lambda,\rho}$, and thus respectively of the form
  \begin{align*}
    t&=m(x_0,m(x_1,m(\ldots,m(x_{n-2},x_{n-1}))))
    \\
    u&=m(x_{f(0)},m(x_{f(1)},m(\ldots,m(x_{f(n-2)},x_{f(n-1)}))))
  \end{align*}
  (up to renaming the variables of~$t$) for some bijection $f:\intset{n}\to\intset{n}$ with $\intset{n}=\set{0,\ldots,n-1}$. From the well-known fact that any such bijection can be obtained as a composite of adjacent transpositions, it is easy to construct a 2-cell $t\Pathto u$ as a composite of rewriting steps $\gamma$ or $\delta$ corresponding to those transpositions. For instance, with $t=m(x_0,m(x_1,x_2))$ and $u=m(x_2,m(x_1,x_0))$, the bijection $f:[3]\to[3]$ is defined by $f(0)=2$, $f(1)=1$ and $f(0)=0$, which can be obtained as the composite of adjacent transpositions $(01)$ and $(12)$, and the corresponding 2-cell is
  \[
    \begin{tikzcd}[arrows={Rightarrow},column sep=15ex]
      m(x_0,m(x_1,x_2))\ar[r,"{\delta(x_0,x_1,x_2)}"]&m(x_1,m(x_0,x_2))\ar[r,"{m(x_1,\gamma(x_0,x_2))}"]&m(x_2,m(x_1,x_0))
    \end{tikzcd} \qedhere
  \]
\end{proof}

\subsection{General coherence}
Finally, we explain how to recover the more general coherence theorem for symmetric monoidal categories. Given a natural number~$n$, we write $\intset{n}=\set{0,\ldots,n-1}$ for the cardinal with $n$ elements.

We define a Lawvere $2$-theory $\mathcal{S}$ as follows. Its 0-cells are natural numbers, as for any Lawvere theory. Given a 0-cell~$n$, a 1-cell $f:n\to 1$ is a list $[f^1,f^2,\ldots,f^{\sizeof{l}}]$ of elements of~$\intset{n}$, sometimes being referred to as \emph{colors}, the natural number $\sizeof{f}$ being the \emph{length} of the list. Since a Lawvere 2-theory is cartesian, we more generally have that a morphism $n\to m$ is a tuple $\tuple{f_0,\ldots,f_{m-1}}$ of $m$ 1-cells $f_i:n\to 1$. Given 1-cells $f:m\to n$ and $g:n\to 1$, their composition is the list computed by taking the list $g$, replacing each color $g^i$ by $f_i$ so that we obtain a list of lists, and then flattening the resulting list, \ie taking the concatenation of all the lists. For instance, we have the composite
\[
  \begin{tikzcd}
    4\ar[r,"\tuple{[1,1],[3,3,2],[2,0,3]}"]&[12ex]3\ar[r,"\tuple{[2,0,2]}"]&[3ex]1
  \end{tikzcd}
  \qquad=\qquad
  \begin{tikzcd}
    4\ar[r,"\tuple{2,0,3,1,1,2,0,3}"]&[10ex]1
  \end{tikzcd}
\]
The 2-cells $\alpha:f\To g:n\to 1$ are color-preserving bijections on $n$ elements, \ie bijections $\alpha:\intset{n}\to\intset{n}$ such that $f^{\alpha(i)}=g^i$ for every $i\in\intset{n}$. More generally, 2-cells $\alpha:f\To g:m\to 1$ are tuples $\tuple{\alpha^0,\ldots,\alpha^{m-1}}$ of 2-cells $\alpha^i:n\to 1$ for $i\in\intset{m}$. The 1-composition of 2-cells is the usual composition of functions, and the 0-composition is left to the reader.

We claim that this 2-theory provides an explicit description of the quotient category for the theory of symmetric monoids, thus establishing a coherence theorem (S2) or (C2). In order to prove this, we will need to relate affine terms to non-affine ones. Informally, every non-affine term can be described as an affine term in which some variables have been identified, and moreover rewriting between non-affine terms can be uniquely lifted to the affine term they originate from. We say that a 1-cell $f:m\to n$ of~$\freecat\SMon_1$ is a \emph{renaming} when it is a tuple of variables, \ie of the form~$\tuple{x_{i_0},\ldots,x_{i_{k-1}}}$.

\begin{lemma}
  \label{smon-affine-lift}
  Consider the \TRS2 $\SMon$. For any term~$t:n\to 1$, there is a natural number~$\tilde n$ and an affine term $\tilde t:\tilde n\to 1$ together with a renaming $f:n\to\tilde n$ such that $t=\tilde t\circ f$:
  \[
    \begin{tikzcd}
      \tilde n\ar[r,dotted,"\tilde t"]&1\ar[d,equals]\\
      n\ar[u,"f"]\ar[r,"t"']&1
    \end{tikzcd}
  \]
  Moreover, for any rewriting zig-zag $p:t\Zzto u$, there is a unique term~$\tilde u$ and rewriting path $p:\tilde t\Zzto\tilde u$ such that $\tilde p\circ f=p$:
  \[
    \begin{tikzcd}
      \tilde n
      \ar[r,bend left,"\tilde t"{name=TT}]
      \ar[r,bend right,dotted,"\tilde u"'{name=UU}]
      &
      1\ar[d,equals]
      \\
      n
      \ar[u,"f"]
      \ar[r,bend left,"t"{name=T}]
      \ar[r,bend right,"u"'{name=U}]
      &1
      \ar[from=T,to=U,Rightarrow,shorten=3pt,"p"']
      \ar[from=TT,to=UU,Rightarrow,dotted,shorten=3pt,"\tilde p"']
    \end{tikzcd}
  \]
  and~$\tilde u$ is affine. Moreover, any equivalence $p\Cohto q$ between $p,q:t\Zzto u$ lifts as an equivalence~$\tilde p\Cohto\tilde q$.
\end{lemma}
\begin{proof}
  Given a term~$t$, we can construct $\tilde t$ by renaming distinctively every occurrence of each variable. For instance, the term~$m(x_1,m(x_0,x_1)):2\to 1$ can be obtained form the affine term $m(x_0,m(x_1,x_2)):3\to 1$ with the renaming $\tuple{x_1,x_0,x_1}:2\to 3$.
  The fact that we can lift rewriting steps follows from the fact that all the rewriting rules in $\SMon_1$ have affine source and targets. And similarly for coherence relations.
\end{proof}

\noindent
The same structure can be identified in $\mathcal{S}$. We say that a 1-cell $f:n\to 1$ is affine, when it consists in a list which does not contain the same element twice. We say that a 1-cell $f:m\to n$ is a \emph{renaming} when it is a tuple of lists of length~1, \ie of the form $\tuple{[i_0],\ldots,[i_{k-1}]}$.

\begin{lemma}
  \label{S-affine-lift}
  Consider the Lawvere 2-theory $\mathcal{S}$. For every 1-cell $t:n\to 1$ there is an affine 1-cell $\tilde t:n\to 1$ together with a renaming $f:n\to\tilde n$ such that $t=\tilde t\circ f$, moreover any 2-cell $p:t\To u$ lifts uniquely as a 2-cell $\tilde p:\tilde t\To\tilde u$ such that $\tilde p\circ f=f$.
\end{lemma}
\begin{proof}
  Given a 1-cell $t:n\to 1$, \ie a list $[i_0,i_1,\ldots,i_{k-1}]$ with values in $\intset{n}$, we can take $\tilde n=k$ and $\tilde t=[0,1,\ldots,k-1]:k\to 1$ together with the renaming $f=\tuple{[i_0],[i_1],\ldots,[i_{k-1}]}:n\to k$. Given a morphism $p:t\To u$, \ie a color-preserving bijection $p:\intset{k}\to\intset{k}$, it is easily shown that the same bijection provides a suitable 2-cell $t\To u$, with $u=[p(0),p(1),\ldots,p(k-1)]$, and that this is the only possible one.
\end{proof}

\noindent
We can finally show the announced coherence theorem (C2) for symmetric monoidal categories:

\begin{theorem}
  \label{smc-S2}
  We have an isomorphism of Lawvere 2-theories $\pgpd\SMon/W\isoto\mathcal{S}$.
\end{theorem}
\begin{proof}
  The two 2-theories have the same 0-cells: the natural numbers, as any Lawvere theory. Since $W$ forms a rewriting system, the 1-cells of $\pgpd\SMon/W$ can be identified to the 1-cells $n\to 1$ of $\pgpd\SMon$ which are in normal form with respect to the rules of~$W$, \ie terms of the form
  \[
    m(x_{i_0},m(x_{i_1},m(\ldots,m(x_{i_{k-2}},x_{i_{k-1}}))))
  \]
  for some $k\in\N$ with $i_j\in\intset{n}$ for $j\in\intset{k}$, and those are clearly in bijection with lists
  \[
    [i_0,i_2,\ldots,i_{n-2},i_{k-1}]
  \]
  of elements of~$\intset{n}$, and the compositions coincide in both categories.
  Given two terms $t,u:n\to 1$, the category $\pgpd\SMon/W(t,u)$ can be identified with $\pgpd\SMon(\nf t,\nf u)$ by \cref{prop:quot-nf}, where $\nf t$ is the normal form of $t$ with respect to the rules of~$W$ and similarly for~$\nf u$. In the following, we implicitly assume that the 1-cells we consider are in normal form and make this identification.
  Given two terms $t,u:n\to 1$, with $t$ affine, we know by \cref{affine-2-cells} that there is at most one 2-cell $t\To u$ in $\pgpd\SMon/W$, and this is the case precisely when $t$ and $u$ have the same variables. If we consider the corresponding lists, still written~$t$ and~$u$, in $\mathcal{S}$, we can observe that there is most one color-preserving bijection between the two lists (because they do not contain the same color twice) and this is the case precisely when they have the same colors. We thus have a bijection between the sets of 2-cells $\pgpd\SMon/W(t,u)\isoto\mathcal{S}(t,u)$ when $t$ and $u$ are affine.
  Finally, given two arbitrary 1-cells~$t,u:n\to 1$, by \cref{smon-affine-lift} the set of 2-cells $\pgpd\SMon/W(t,u)$ can be described as
  \[
    \pgpd\SMon/W(t,u)
    =
    \bigsqcup_{\substack{u':n\to 1\\\text{$u'$ affine}}}\setof{p:t\To u'}{u'\circ f=u}
  \]
  Similarly, by \cref{S-affine-lift} the set of 2-cells $\mathcal{S}(t,u)$ can be described as
  \[
    \mathcal{S}(t,u)
    =
    \bigsqcup_{\substack{u':n\to 1\\\text{$u'$ affine}}}\setof{p:t\To u'}{u'\circ f=u}
  \]
  Those sets are isomorphic as disjoint unions of 2-cells whose source is affine, and the bijection can be checked to be compatible with composition.
\end{proof}

\noindent
Alternatively, the proof could be done by constructing explicitly a presentation of the Lawvere 2-theory $\mathcal{S}$, following previous work in the setting of monoidal categories~\cite{guiraud2012coherence,lafont2003towards}. The proof scheme used above however has the advantage of being simpler to check, and we believe that it is quite general; in particular, extensions to the study of coherence of fundamental structures such as cartesian closed categories should be developed in subsequent works.





\section{Future works}
\label{sec:conclusion}
We claim that the framework developed here applies to a wide variety of algebraic structures, which will be explored in subsequent work. In fact, the full generality of the framework was not needed for (symmetric) monoidal categories, since the rules of the corresponding theory never need to duplicate or erase variables (and, in fact, those can be handled by traditional polygraphs~\cite{guiraud2012coherence,lafont2003towards}). This is however, needed for the case of rig categories~\cite{laplaza1972coherence}, which feature two monoidal structures $\oplus$ and $\otimes$, and natural isomorphisms such as
\[
  \delta_{x,y,z}:x\otimes(y\oplus z)\to(x\otimes y)\oplus(x\otimes z)
\]
where $x$ occurs twice in the target, generalizing the laws for rings. Those were a motivating example for this work, and we will develop elsewhere a proof of coherence of those structures based on our rewriting framework, as well as related approaches on the subject~\cite{bonchi2022tape,bonchi2023deconstructing}.


The importance of the notion of polygraph can be explained by the fact that they are the cofibrant objects in a model structure on $\omega$-categories~\cite{polygraphs,lafont2010folk}. It would be interesting to develop a similar point of view for higher-dimensional term rewriting systems: a first step in this direction is the model structure developed in~\cite{yanofsky2001coherence}.

\bibliographystyle{alphaurl}
\bibliography{papers}

\end{document}